\documentclass[a4paper, 12point]{article}
\usepackage{}
\newcommand{\be}{\begin{equation}}
\newcommand{\ee}{\end{equation}}
\newcommand{\bea}{\begin{eqnarray}}
\newcommand{\eea}{\end{eqnarray}}

\usepackage{graphicx,epsfig}
\usepackage{enumerate}
\usepackage{amssymb}
\usepackage{mathrsfs}
\usepackage{amsthm}
\usepackage{amsmath}
\usepackage{setspace}
\usepackage{endnotes}
\newtheorem{thm}{Theorem}[section]
\theoremstyle{definition}

\theoremstyle{lemma}
\newtheorem{lem}{Lemma}[section]
\theoremstyle{example}

\newtheorem{rmrk}{Remark}[section]
\theoremstyle{illustration}

\theoremstyle{proposition}

\theoremstyle{corollary}

\numberwithin{equation}{section}
\begin{document}
\date{}
\title{\textbf{Some $B$-Difference Sequence Spaces Derived by Using Generalized Means and Compact Operators}}
\author{Amit Maji\footnote{ Corresponding author's, e-mail:
amit.iitm07@gmail.com}, Atanu Manna \footnote{Author's e-mail: atanumanna@maths.iitkgp.ernet.in}, P. D. Srivastava \footnote{Author's e-mail: pds@maths.iitkgp.ernet.in}\\
\textit{\small{Department of Mathematics, Indian Institute of Technology, Kharagpur}} \\
\textit{\small{Kharagpur 721 302, West Bengal, India}}}
\maketitle
\vspace{20pt}
\begin{center}\textbf{Abstract}\end{center}
This paper presents new sequence spaces $X(r, s, t, p ; B)$ for
$X \in \{ l_\infty(p), c(p), c_0(p), l(p)\}$ defined by using
generalized means and difference operator. It is shown that these
spaces are complete paranormed spaces and the spaces $X(r, s, t, p ; B)$ for
$X \in \{ c(p), c_0(p), l(p)\}$ have Schauder basis. Furthermore,
the $\alpha$-, $\beta$-, $\gamma$- duals of these sequence spaces are computed
 and also obtained necessary and sufficient conditions for some matrix transformations from $X(r, s, t, p ;B)$ to $X$. Finally,
 some classes of compact operators on the space $l_p(r, s, t ;B)$ are characterized by using the Hausdorff measure of noncompactness.
\section{Introduction}
The study of sequence spaces has importance in the several branches of analysis, namely, the structural theory of topological vector spaces, summability theory, Schauder basis theory etc. Besides this, the theory of sequence spaces is a powerful tool for obtaining some topological and geometrical results by using Schauder basis.

Let $w$ be the space of all real or complex sequences $x=(x_n), n\in \mathbb{N}_0$. For an infinite matrix $A$ and a sequence space $\lambda$, the matrix domain of $A$, denoted by $\lambda_{A}$, is defined as $\lambda_A=\{x\in w: Ax\in \lambda\}$ \cite{WIL}. Some basic methods, which are used to determine the topologies, matrix transformations and inclusion relations on sequence spaces can also be applied to study the matrix domain $\lambda_A$. In recent times, there is an approach of forming a new sequence space by using a matrix domain of a suitable matrix and characterize the matrix mappings between these sequence spaces.
\par Let $(p_k)$ be a bounded sequence of strictly positive real numbers such that $H=\displaystyle\sup_k p_k$ and $M=\max \{1, H\}$. The linear spaces $c(p), c_0(p)$, $l_\infty(p)$ and $l(p)$ are studied by Maddox \cite{MAD}, where
\begin{align*}
&c(p)=\Big\{x=(x_k)\in w:  \displaystyle \lim_{k\rightarrow\infty}|x_k-l|^{p_k}=0 \mbox{~~for some~} l\in \mathbb{C}\Big\},\\
&c_0(p)=\Big\{x=(x_k)\in w:  \displaystyle \lim_{k\rightarrow\infty}|x_k|^{p_k}=0 \Big\},\\
&l_\infty(p)=\Big\{x=(x_k)\in w:  \displaystyle \sup_{k\in \mathbb{N_{\rm 0}}}|x_k|^{p_k}<\infty\Big\} \mbox{~and~}\\
&l(p)=\Big\{x=(x_k)\in w:  \displaystyle \sum_{k=0}^{\infty}|x_k|^{p_k}<\infty\Big\}.
\end{align*}
The linear spaces $c(p), c_0(p)$, $l_\infty(p)$ are complete with the paranorm $g(x) = \displaystyle \sup_{k}|x_k|^{p_k \over M}$ if and only if $\inf p_k>0$ for all $k$ while $l(p)$ is complete with the paranorm $\tilde g(x) =
\displaystyle \Big(\sum_{k}|x_k|^{p_k }\Big)^{1 \over M}$.
Recently, several authors introduced new sequence spaces by using matrix domain. For example, Ba\c{s}ar et al. \cite{BAS2} studied the space $bs(p)=[l_\infty(p)]_S$, where $S$ is the summation matrix. Altay and Ba\c{s}ar \cite{ALT1} studied the sequence spaces $r^t(p)$ and $r_{\infty}^t(p)$, which consist of all sequences whose Riesz transform are in the spaces $l(p)$ and $l_\infty(p)$ respectively, i.e., $r^t(p)=[l(p)]_{R^t}$ and $r_{\infty}^t(p)= [l_{\infty}(p)]_{R^t}$. Altay and Ba\c{s}ar  also studied the sequence spaces $r_{c}^t(p)= [c(p)]_{R^t}$ and $r_{0}^t(p)= [c_0(p)]_{R^t}$ in \cite{ALT}.

Kizmaz first introduced and studied the difference sequence space in \cite{KIZ}.
Later on, many authors including Ahmad and Mursaleen \cite{AHM}, \c{C}olak and Et \cite{COL}, Ba\c{s}ar and Altay \cite{ALT}, etc.
studied new sequence spaces defined by using difference operator. Using Euler mean of order $\alpha$, $0< \alpha< 1$ and difference operator, Karakaya and Polat
introduced the paranormed sequence spaces $e_0^\alpha(p; \Delta),e^\alpha_c(p; \Delta)$ and $e_{\infty}^\alpha(p; \Delta)$ in \cite{KAR}.  Mursaleen and Noman \cite{MUR} introduced a sequence
space of generalized means, which includes most of the earlier known
sequence spaces. But till $2011$, there was no such literature available in which a sequence space is generated by
combining both the weighted mean and the difference operator. This
was first initiated by Polat et al. \cite{POL}.
Later on, Demiriz and \c{C}akan \cite{DEM} introduced the new paranormed difference sequence spaces $\lambda(r', s', p; \Delta)$ for $\lambda \in \{l_\infty(p), c(p), c_0(p), l(p)\}$.

Quite recently, Ba\c{s}arir and Kara \cite{BASA1} introduced and studied the $B$-difference sequence space $l(r', s', p; B)$ defined as
$$l(r', s', p; B)=\Big\{x\in w: (G(r', s').B) x \in l(p)\Big\},$$
where $r'=(r_n') ,s'=(s_n')$ are non zero sequences and the matrices $G(r', s')=(g_{nk})$, $B=B(u,v)=(b_{nk}), u, v \neq 0$ are defined by
\begin{align*}
g_{nk} &= \left\{
\begin{array}{ll}
    r'_ns'_k & \quad \mbox{~if~} 0\leq k \leq n,\\
    0 & \quad \mbox{~if~} k > n
\end{array}\right.& \mbox{and~}
b_{nk}& = \left\{
\begin{array}{ll}
    0 & \quad \mbox{~if~} 0\leq k <n-1,\\
    v & \quad \mbox{~if~}  k =n-1 \\
    u & \quad \mbox{~if~}  k =n\\
    0 & \quad \mbox{~if~} k> n.
\end{array}\right.
\end{align*}
By using matrix domain, one can write $l(r', s', p; B)=[l{(p)}]_{G(r', s'). B }$.

The aim of this present paper is to introduce the sequence spaces $X(r,s,t,p; B)$ for $X \in \{  l_\infty(p), c(p), $ $c_0(p), l(p)\}$. We have shown that these sequence spaces are complete paranormed sequence spaces under some suitable paranorm. Some topological results and the $\alpha$-, $\beta$-, $\gamma$- duals of these spaces are obtained. A characterization of some matrix transformations between these new sequence spaces is established. We also give a characterization of some classes of compact operators on the space $l_p(r, s, t ;B)$ by using the Hausdorff measure of noncompactness.
\section{Preliminaries}
Let $l_\infty, c$ and $c_0$ be the spaces of all bounded, convergent and null sequences $x=(x_n)$ respectively, with the norm $\|x\|_\infty=\displaystyle\sup_{n}|x_n|$. Let  $bs$ and $cs$ be the sequence spaces of all bounded and convergent series respectively. We denote by $e=(1, 1, \cdots)$ and $e_{n}$ for the sequence whose $n$-th term is $1$ and others are zero and $\mathbb{{N_{\rm 0}}}=\mathbb{N}\cup \{0\}$, where $\mathbb{N}$ is the set of all natural numbers.
A sequence $(b_n)$ in a normed linear space $(X,
\|.\|)$ is called a Schauder basis for $X$ if for every $x\in X$
there is a unique sequence of scalars $(\mu_n)$ such that
\begin{center}
$\Big\|x-\displaystyle\sum_{n=0}^{k}\mu_nb_n\Big\|\rightarrow0$ as $k\rightarrow\infty$,
\end{center}
i.e., $x=\displaystyle\sum_{n=0}^{\infty}\mu_nb_n$ \cite{WIL}.\\
For any subsets $U$ and $V$ of $w$, the multiplier space $M(U, V)$ of $U$ and $V$ is defined as
\begin{center}
$M(U, V)=\{a=(a_n)\in w : au=(a_nu_n)\in V ~\mbox{for all}~ u\in U\}$.
\end{center}
In particular,
\begin{center}
$U^\alpha= M(U, l_1)$, $U^\beta= M(U, cs)$ and $U^\gamma= M(U, bs)$
\end{center} are called the $\alpha$-, $\beta$- and $\gamma$- duals of $U$ respectively \cite{WIL}.

Let $A=(a_{nk})_{n, k}$ be an infinite matrix
with real or complex entries $a_{nk}$. We write $A_n$ as the
sequence of the $n$-th row of $A$, i.e.,
$A_n=(a_{nk})_{k}$ for every $n$.
For $x=(x_n)\in w$, the $A$-transform of $x$ is defined as the
sequence $Ax=((Ax)_n)$, where $A_n(x)=(Ax)_n=\displaystyle\sum_{k=0}^{\infty}a_{nk}x_k$,
provided the series on the right side converges for each $n$. For any two sequence spaces $U$ and $V$, we denote by $(U, V)$, the class of all infinite matrices $A$ that map from $U$ into $V$. Therefore $A\in (U, V)$ if and only if $Ax=((Ax)_n)\in V$ for all $x\in U$. In other words, $A\in (U, V)$ if and only if $A_n \in U^\beta$ for all $n$ \cite{WIL}.
The theory of $BK$ spaces is the most powerful tool in the characterization of matrix transformations between sequence spaces. A sequence space $X$ is called a $BK$ space if it is a Banach space with continuous coordinates $p_n: X\rightarrow \mathbb{K}$, where $\mathbb{K}$ denotes the real or complex field and $p_n(x)=x_n$ for all $x=(x_n)\in X$ and each $n\in \mathbb{N}_{0}$.
The space $l_1$ is a $BK$ space with the usual norm defined by $\|x\|_{l_1}=\displaystyle\sum_{k=0}^{\infty}|x_k|$. An infinite matrix $T=(t_{nk})_{n,k}$ is called a triangle if $t_{nn}\neq 0$ and $t_{nk}=0$ for all $k>n$.
Let $T$ be a triangle and $X$ be a $BK$ space. Then $X_T$ is also a $BK$ space with the norm given by $\|x\|_{X_T}= \|Tx\|_X$ for all $x\in X_T$ \cite{WIL}.

\section{Sequence spaces $X(r, s, t, p; B)$ for $X \in \{ l_{\infty}(p), c(p), c_{0}(p), l(p)\}$}
In this section, we first begin with the notion of generalized means given by Mursaleen et al. \cite{MUR}. \\
We denote the sets $\mathcal{U}$ and $\mathcal{U}_{0}$ as
\begin{center}
$ \mathcal{U} = \Big \{ u =(u_{n})_{n=0}^{\infty} \in w: u_{n} \neq
0~~ {\rm for~ all}~~ n \Big \}$ and $ \mathcal{U_{\rm 0}} = \Big \{ u
=(u_{n})_{n=0}^{\infty} \in w: u_{0} \neq 0 \Big \}.$
\end{center}
Let $r, t \in \mathcal{U}$ and $s \in \mathcal{U}_{0}$. The sequence $y=(y_{n})$ of generalized means of a sequence $x=(x_{n})$ is defined
by $$ y_{n}= \frac{1}{r_{n}}\sum_{k=0}^{n} s_{n-k}t_{k}x_{k} \qquad (n \in \mathbb{N_{\rm 0}}).$$
The infinite matrix $A(r, s, t)$ of generalized means is defined by

$$(A(r,s,t))_{nk} = \left\{
\begin{array}{ll}
    \frac{s_{n-k}t_{k}}{r_{n}} & \quad 0\leq k \leq n,\\
    0 & \quad k > n.
\end{array}\right. $$

Since $A(r, s, t)$ is a triangle, it has a unique inverse and the
inverse is also a triangle \cite{JAR}. Take $D_{0}^{(s)} =
\frac{1}{s_{0}}$ and

$ D_{n}^{(s)} =
\frac{1}{s_{0}^{n+1}} \left|
\begin{matrix}
    s_{1} & s_{0} &  0 & 0 \cdots & 0 \\
    s_{2} & s_{1} & s_{0}& 0 \cdots & 0 \\
    \vdots & \vdots & \vdots & \vdots    \\
    s_{n-1} & s_{n-2} & s_{n-3}& s_{n-4} \cdots & s_0 \\
      s_{n} & s_{n-1} & s_{n-2}& s_{n-3} \cdots & s_1
\end{matrix} \right| \qquad \mbox{for} $ $n=1, 2, 3, \ldots$\\

Then the inverse of $A(r, s, t)$ is the triangle $\tilde{B}= (\tilde{b}_{nk})_{n, k}$, which is defined as
$$\tilde{b}_{nk} = \left\{
\begin{array}{ll}
    (-1)^{n-k}~\frac{D_{n-k}^{(s)}}{t_{n}}r_{k} & \quad 0\leq k \leq n,\\
    0 & \quad k > n.
\end{array}\right. $$
Throughout this paper, we consider $p=(p_k)$ be a bounded sequence of strictly positive real numbers such that $H=\displaystyle\sup_kp_k$ and $M=\max\{1, H\}$. \\
We now introduce the sequence spaces $X(r, s, t, p; B)$ for $X \in \{ l_{\infty}(p), c(p), c_{0}(p), l(p)\}$, combining both the generalized means and the matrix $B(u,v )$ as
$$ X(r, s, t, p; B)= \Big \{ x \in w :  y=A(r,s,t; B)x \in X \Big  \},$$
where $ y=(y_{k})$ is $A(r,s,t; B)$-transform of a sequence $x=(x_{k})$, i.e.,
\begin{equation}{\label{3eq}}
 y_n = \frac{1}{r_n}  \Big(\displaystyle \sum_{k=0}^{n-1} (s_{n-k}t_k u + s_{n-k-1}t_{k+1}v)x_k + s_0t_nux_n\Big) ,~~~ n \in \mathbb{N}_{0},
\end{equation}
where we mean $\displaystyle \sum_{n}^{m}=0$ for $n>m$. By using matrix domain, we can write $X(r, s,t, p; B)=  X_{A(r, s,t,p; B)}=\{x \in w :A(r, s,t; B)x\in X\}$, where $A(r, s,t; B)= A(r, s,t). B$, product of two triangles $A(r, s,t)$ and $B(u,v)$. For $X=l_p$, $p\geq 1$, we write $X(r, s, t, p; B)$ as $l_p(r, s, t; B)$.

These sequence spaces include many known sequence spaces studied by several authors. For examples,
\begin{enumerate}[I.]
\item if $r_{n}=\frac{1}{r'_{n}}$, $t_{n}=s'_{n}$, $s_{n}=1$ $\forall n$, then the sequence spaces $ l(r, s,t, p; B)$ reduce to $ l(r', s', p; B)$
 studied by Ba\c{s}arir and Kara \cite{BASA1}.
\item if $r_{n}=\frac{1}{r'_{n}}$, $t_{n}=s'_{n}$, $s_{n}=1$ $\forall n$, $u=1$ and $v=-1$ then the sequence spaces $ X(r, s,t, p; B)$ reduce to $ X(r', s', p; \Delta)$ for $X \in \{ l_{\infty}(p), c(p), c_{0}(p), l(p)\}$ studied by Demiriz and \c{C}akan \cite{DEM}.

\item if $r_{n}=\frac{1}{n!},$ $t_{n}=\frac{\alpha^n}{n!}$, $s_{n}=\frac{(1-\alpha)^n}{n!}$, where $0<\alpha<1$, $u=1$, $v=-1$, then the sequence spaces $ X(r, s,t, p; B)$ for $X\in\{l_{\infty}(p), c(p), c_{0}(p)\}$, reduce to $e^\alpha_\infty(p; \Delta)$, $e_c^\alpha(p; \Delta)$, $e^\alpha_0(p; \Delta)$ respectively studied by Karakaya and Polat \cite{KAR}.
\item if $\lambda=(\lambda_k)$ be a strictly increasing sequence of positive real numbers tending to infinity such that $ r_{n}=\lambda_n,$ $t_{n}=\lambda_n-\lambda_{n-1}$, $s_{n}=1$ $\forall n$, then $l_p(r, s,t; B)$ reduces to the sequence space $l_p^\lambda(B)$, where the sequence space $l_p^\lambda$ is introduced and studied by Mursaleen and Noman \cite{MUR5}.
\item if $r_{n}=n+1,$ $t_{n}={1+ \alpha^n}$, where $0<\alpha<1$, $s_{n}=1$ $\forall n$, $u=1$, $v=-1$, then the sequence space $l_p(r, s, t; B)$ reduces to the sequence space $a_p^{\alpha}(\Delta)$ studied by Demiriz and \c{C}akan \cite{DEM1}.
\end{enumerate}

\section{Main results}
Throughout the paper, we denote the sequence spaces $X(r, s, t, p; B)$ as $l(r, s, t, p; B)$, $c_0(r, s, t, p; B)$, $c(r, s, t, p; B)$ and $l_{\infty}(r, s, t, p; B)$
when $X =l(p),c_0(p), c(p)$ and $l_{\infty}(p)$ respectively. Now, we begin with some topological results of the new sequence spaces.
\begin{thm}
$(a)$ The sequence space $l(r, s, t, p; B)$ is a complete
linear metric space paranormed by $\tilde{h}$ defined as
\begin{center}
$\tilde{h}(x)=\Big(\displaystyle\sum_{n=0}^{\infty}\Big|\frac{1}{r_n}  \Big(\displaystyle \sum_{k=0}^{n-1} (s_{n-k}t_k u + s_{n-k-1}t_{k+1}v)x_k + s_0t_nux_n\Big) \Big|^{p_n} \Big)^\frac{1}{M}$.
\end{center}
$(b)$ The sequence spaces $X(r,s, t, p ; B)$ for $X\in
\{l_{\infty}(p)$, $ c(p), c_{0}(p)\}$ are complete linear metric
spaces paranormed by $h$ defined as
\begin{center}
$h(x)=\displaystyle\sup_{n\in \mathbb{N}_{0}}\Big|\frac{1}{r_n}  \Big(\displaystyle \sum_{k=0}^{n-1} (s_{n-k}t_k u + s_{n-k-1}t_{k+1}v)x_k + s_0t_nux_n\Big) \Big| ^\frac{p_n}{M}$.
\end{center}
$(c)$ The sequence space $\ell_p(r,s,t; B)$, $1\leq p < \infty$ is a $BK$ space with the norm given by
\begin{center}
$ \| x \|_{l_p(r,s,t; B)} = \| y\|_{l_p}$,
\end{center}
where $y=(y_k)$ is defined in (\ref{3eq}).
\end{thm}

\begin{proof}We prove only the part $(a)$ of this theorem. In
a similar way, we can prove the other parts.
\\Let
$x$, $y\in l(r, s, t, p; B)$. Using Minkowski's inequality
\begin{align}
&\Big(\displaystyle\sum_{n=0}^{\infty}\Big|\frac{1}{r_n}  \Big(\displaystyle \sum_{k=0}^{n-1} (s_{n-k}t_k u + s_{n-k-1}t_{k+1}v)(x_k+ y_k) + s_0t_nu(x_n+ y_n)\Big) \Big|^{p_n} \Big)^\frac{1}{M} \nonumber \\
&~~~~~~\leq  \Big(\displaystyle\sum_{n=0}^{\infty}\Big|\frac{1}{r_n}  \Big(\displaystyle \sum_{k=0}^{n-1} (s_{n-k}t_k u + s_{n-k-1}t_{k+1}v)x_k + s_0t_nux_n\Big) \Big|^{p_n} \Big)^\frac{1}{M} \nonumber \\
&~~~~~~~~~~~~~~+ \Big(\displaystyle\sum_{n=0}^{\infty}\Big|\frac{1}{r_n}  \Big(\displaystyle \sum_{k=0}^{n-1} (s_{n-k}t_k u + s_{n-k-1}t_{k+1}v)y_k + s_0t_nuy_n\Big)\Big|^{p_n} \Big)^\frac{1}{M}.
\end{align}
So, we have $x+y \in l(r, s, t, p; B)$.
Let $\alpha$ be any scalar. Since $|\alpha|^{p_k}\leq \max\{1, |\alpha|^M\}$, we have
 $\tilde{h}(\alpha x)\leq \max\{1, |\alpha|\}\tilde{h}(x)$. Hence $\alpha x \in l(r, s, t, p; B)$.
It is trivial to show that $\tilde{h}(\theta)=0$, $\tilde{h}(-x)=\tilde{h}(x)$ for all $x\in l(r, s, t, p; B)$
 and subadditivity of $\tilde{h}$, i.e., $\tilde{h}(x+y)\leq \tilde{h}(x)+ \tilde{h}(y)$ follows from $(4.1)$.\\
 Next we show that the scalar multiplication is continuous.
Let $(x^m)$ be a sequence in $l(r, s,
t, p; B)$, where $x^{m}= (x_k^{m})=(x_0^{m}, x_1^{m}, x_2^{m}, \ldots )$ $\in l(r, s, t, p; B)$ for each $m\in \mathbb{N}_0$ such that $\tilde{h}(x^{m}-x)\rightarrow0$ as
$m\rightarrow\infty$ and $(\alpha_m)$ be a sequence of scalars such
that $\alpha_m\rightarrow \alpha$ as $m\rightarrow\infty$. Then
$\tilde{h}(x^{m})$ is bounded that follows from the following
inequality
\begin{center}
$\tilde{h}(x^{m})\leq \tilde{h}(x)+ \tilde{h}(x-x^{m})$.
\end{center}
Now consider
\begin{align*}
\tilde{h}(\alpha_mx^{m}-\alpha x) & = \bigg(\displaystyle\sum_{n=0}^{\infty}\Big|\frac{1}{r_n}  \Big(\displaystyle \sum_{k=0}^{n-1} (s_{n-k}t_k u + s_{n-k-1}t_{k+1}v)(\alpha_mx^{m}_k-\alpha x_k) + s_0t_nu(\alpha_mx^{m}_n-\alpha x_n)\Big)\Big|^{p_n}\bigg)^\frac{1}{M} \\
& \leq |\alpha_m-\alpha| \tilde{h}(x^{m}) + |\alpha| \tilde{h}(x^{m}-x)\rightarrow 0 \mbox{~as~} m\rightarrow\infty.
\end{align*}This shows that the scalar multiplication is continuous. Hence $\tilde{h}$ is a paranorm on the space $l(r, s, t, p; B)$.\\
Now we prove the completeness of the space $l(r, s, t, p;
B)$ with respect to the paranorm $\tilde{h}$. Let
$(x^{m})$ be a Cauchy sequence in $l(r,
s, t, p; B)$. So for every $\epsilon>0$ there is a
$n_0\in \mathbb{N}$ such that
\begin{center}
$\tilde{h}(x^{m}-x^{l})< \frac{\epsilon}{2}$ ~ for all $m, l\geq
n_0$.
\end{center}
Then by definition for each $n \in \mathbb{N}_0$, we have
\begin{equation}
\displaystyle \bigg|(A(r, s, t; B)x^{m})_n -(A(r, s, t; B)x^{l})_n \bigg| \leq
 \bigg(\displaystyle\sum_{n=0}^{\infty}\bigg|(A(r, s, t; B)x^{m})_n -(A(r, s, t; B)x^{l})_n \bigg|^{p_n} \bigg)^{1 \over M} <
\frac{\epsilon}{2}
\end{equation}
~\mbox{for all}~ $m, l\geq n_0,$ which implies that the
sequence $((A(r, s, t; B)x^{m})_n)$ is a
Cauchy sequence of scalars for each fixed $n \in \mathbb{N}_0$
and hence converges for each $n$. We write
$$ \displaystyle\lim_{m \rightarrow \infty} (A(r, s, t; B)x^{m})_n = (A(r, s, t; B)x)_n \quad ( n \in \mathbb{N}_0).$$
Now taking $l \rightarrow \infty$ in (4.2), we obtain
$$\bigg(\displaystyle\sum_{n=0}^{\infty}\bigg|(A(r, s, t; B)x^{m})_n -(A(r, s, t; B)x)_n \bigg|^{p_n} \bigg)^{1 \over M} <
\epsilon$$
~\mbox{for all}~ $m \geq n_0$ and each fixed $n \in \mathbb{N}_0$. Thus $(x^{m})$ converges to $x$ in $l(r, s, t, p; B)$
with respect to $\tilde{h}$.\\
 To show  $x\in l(r, s, t, p; B)$, we take
\begin{align*}
&\Big(\displaystyle\sum_{n=0}^{\infty}\Big|\frac{1}{r_n}  \Big(\displaystyle \sum_{k=0}^{n-1} (s_{n-k}t_k u + s_{n-k-1}t_{k+1}v)x_k + s_0t_nux_n\Big)  \Big|^{p_n}\Big)^\frac{1}{M}\\
&= \Big(\displaystyle\sum_{n=0}^{\infty}\Big|\frac{1}{r_n}  \Big(\displaystyle \sum_{k=0}^{n-1} (s_{n-k}t_k u + s_{n-k-1}t_{k+1}v)(x_k- x_{k}^{m}+ x_{k}^{m}) + s_0t_nu(x_n- x_{n}^{m}+ x_{n}^{m})\Big)\Big|^{p_n}\Big)^\frac{1}{M}\\
&\leq \tilde{h}(x -x^{m}) + \tilde{h}(x^{m}),
\end{align*}
which is finite for all $m \geq n_0$. Therefore $x\in l(r, s, t, p; B)$. This completes the proof.
\end{proof}
\begin{thm}
The sequence spaces $X(r,s,t, p; B)$ for $X \in \{
l_{\infty}(p)$, $ c(p), c_{0}(p), l(p)\}$ are linearly isomorphic to
the spaces $X \in \{ l_{\infty}(p)$, $ c(p), c_{0}(p), l(p)\}$
respectively, i.e., $l_{\infty}(r, s, t, p ; B) \cong
l_{\infty}(p)$,  $c(r, s, t, p ; B) \cong c(p)$, $c_{0}(r, s,
t, p ; B) \cong c_{0}(p)$ and $l(r,s,t, p; B) \cong l(p)$.
\end{thm}
\begin{proof}
We prove the theorem only for the case when $X=l(p)$. To prove
this we need to show that there exists a bijective linear map from
$l(r, s, t, p ; B)$ to $l(p)$. Now we define a map $T:l(r, s, t,
p ; B) \rightarrow l(p)$ by $x \mapsto Tx = y=(y_n)$, where
\begin{center}\label{k1}
$y_n = \frac{1}{r_n}\Big(\displaystyle \sum_{k=0}^{n-1} (s_{n-k}t_k u + s_{n-k-1}t_{k+1}v)x_k + s_0t_nux_n\Big),~~~ n \geq 0.$
\end{center}
The linearity of $T$ is
trivial. It is easy to see that $Tx=0$ implies $x=0$. Thus $T$ is
injective. To prove $T$ is surjective, let $y \in l(p)$. Since $y=
(A(r, s, t). B)x$, i.e., $x= (A(r, s, t).
B)^{-1}y=B^{-1}.A(r, s, t)^{-1}y $, so we can get a sequence
$x=(x_{n})$ as
\begin{equation}
  x_n = \displaystyle\sum_{j=0}^{n}\sum_{k=j}^{n} (-1)^{k-j} \frac{(-v)^{n-k}}{u^{n-k+1}}\frac{D_{k-j}^{(s)}}{t_{k}}r_jy_j, \quad n \in \mathbb{N}_0.
\end{equation}
 Then
\begin{center}
$\tilde{h}(x)=\Big(\displaystyle\sum_{n=0}^{\infty}\Big|\frac{1}{r_n}  \Big(\displaystyle \sum_{k=0}^{n-1} (s_{n-k}t_k u + s_{n-k-1}t_{k+1}v)x_k + s_0t_nux_n\Big)\Big|^{p_n}\Big)^\frac{1}{M}=\Big(\displaystyle\sum_{n=0}^{\infty}\big|y_n\big|^{p_n}\Big)^\frac{1}{M}=\tilde{g}(y)<\infty$.
\end{center}
Thus $x \in l(r, s, t, p; B)$ and this shows that $T$ is
surjective. Hence $T$ is a linear bijection from $l(r, s, t, p;
B)$ to $l(p)$. Also $T$ is paranorm preserving. This completes
the proof.
\end{proof}

Since $X(r,s,t, p; B) \cong X$ for $X \in \{c_{0}(p), c(p), l(p)\}$, the Schauder bases of the sequence spaces $X(r,s,t, p; B)$ are the inverse image of the bases of $X$ for $X \in \{c_{0}(p), c(p), l{(p)}\}$. So, we have the following theorem without proof.
\begin{thm}
Let $\nu_k=(A(r,s, t; B)x)_k$, $k\in \mathbb{N_{\rm
0}}$. For each $j \in \mathbb{N_{\rm 0}}$, define the sequence
$b^{(j)}=(b_{n}^{(j)})_{n\in\mathbb{N_{\rm 0}}}$ of the elements of
the space $c_0(r, s, t, p; B)$ as
$$
b_n^{(j)} = \left\{
\begin{array}{ll}
 \displaystyle\sum_{k=j}^{n} (-1)^{k-j} \frac{(-v)^{n-k}}{u^{n-k+1}}\frac{D_{k-j}^{(s)}}{t_{k}}r_j & \mbox{if}~~   0\leq j \leq n \\
0 & \mbox{if} ~~ j>n.
\end{array}
\right.
$$ and
$$ b_n^{(-1)} = \displaystyle \sum_{j=0}^{n}\sum_{k=j}^{n} (-1)^{k-j} \frac{(-v)^{n-k}}{u^{n-k+1}}\frac{D_{k-j}^{(s)}}{t_{k}}r_j.  $$
Then the followings are true:\\
 $(i)$ The sequence $(b^{(j)})_{j=0}^{\infty}$ is a basis for the space $X(r,s,t, p; B)$ for $X \in \{c_{0}(p), l(p)\}$ and any
 $x\in  X(r,s,t, p; B)$ has a unique representation of the form
\begin{center}
$x=\displaystyle\sum_{j=0}^{\infty}\nu_jb^{(j)}$.
\end{center}
$(ii)$ The set $(b^{(j)})_{j=-1}^{\infty}$ is a basis for the space $c(r,
s,t,p; B)$ and any $x\in c(r,
s,t,p; B)$ has a unique
representation of the form
\begin{center}
$x=\ell b^{(-1)}+ \displaystyle\sum_{j=0}^{\infty}(\nu_j-\ell)b^{(j)}$,
\end{center}where $\ell=\displaystyle\lim_{n\rightarrow\infty}(A(r,s, t;B)x)_n$.
\end{thm}
\begin{rmrk}
In particular, if we choose $r_{n}=\frac{1}{r_{n}'}$,
$t_{n}=s_{n}'$, $s_{n}=1$ $\forall n$, then the sequence space $ l(r,
s,t, p; B)$ reduces to $ l(r', s', p; B)$ \cite{BASA1}. With this choice of $s_{n}$,
we have $D_{0}^{(s)} =D_{1}^{(s)}=1$ and $D_{n}^{(s)} =0 $ for $ n
\geq 2$. Thus the sequences $b^{(j)} = ( b_{n}^{(j)})_{n \in
\mathbb{N_{\rm 0}}}$  for $j = 0, 1, \ldots$ reduce to
\begin{displaymath}
  b_{n}^{(j)}  = \left\{
     \begin{array}{ll}
        {  \frac{(-1)^{n-k}}{r_j'}\Big(\frac{v^{n-j}}{u^{n-j+1}}\frac{1}{s_j'} + \frac{v^{n-j-1}}{u^{n-j}}\frac{1}{s_{j+1}'}\Big) }& \mbox{if} \quad 0\leq j < n \\
        \frac{1}{u}\frac{1}{r_{n}'s_{n}'}              & \mbox{if}  \quad  j=n\\
        0                                 & \mbox{if} \quad j > n.
     \end{array}
   \right.
\end{displaymath}
 The sequence $( b^{(j)} )$ is a Schauder basis for the space $ l(r',
s', p; B)$ studied in \cite{BASA1}.
\end{rmrk}

\subsection{The $\alpha$-, $\beta $-, $\gamma$-duals of $X(r,s, t, p; B)$ for
$X\in\{l_\infty(p), c(p), c_0(p), l(p)\}$}
In $1999$, K. G. Grosse-Erdmann \cite{GRO} has characterized the matrix transformations between the sequence spaces of Maddox, namely, $l_\infty(p), c(p), c_0(p)$ and $ l(p)$.
To compute the $\alpha$-, $\beta $-, $\gamma$-duals of $X(r,s, t, p; B)$ for
$X\in\{l_\infty(p), c(p), c_0(p), l(p)\}$ and to characterize the classes of some matrix mappings between these spaces, we list the following conditions.\\
Let $L$ denotes a natural number, $F$ be a nonempty finite subset of $\mathbb{N}$ and $A=(a_{nk})_{n,k}$ be an
infinite matrix.  We consider $p_k'=\frac{p_k}{p_k-1}$ for $1< p_k < \infty$.
\begin{align}
&\displaystyle\sup_{F}\displaystyle\sup_{k}\Big|\displaystyle\sum_{n\in F} a_{nk}
\Big|^{p_k}<\infty\\
& \displaystyle\sup_{F}
\displaystyle\sum_{k}\Big|\displaystyle\sum_{n\in F} a_{nk}
L^{-1} \Big|^{p_k'}<\infty
\mbox{~for some~} L\\
& \displaystyle\lim_{n}a_{nk}=0 \mbox{~for every~} k\\
&\displaystyle\sup_{n}\sup_{k}\displaystyle|a_{nk}L|^{p_k}<\infty \mbox{~for all}~ L\\
&\displaystyle\sup_{n}\sum_{k}\displaystyle|a_{nk}L|
^{{p_k}'}<\infty \mbox{~for all}~ L\\
&\displaystyle\sup_{n}\sup_{k}\displaystyle|a_{nk}|^{p_k}<\infty \\
& \exists~ (\alpha_k) \displaystyle\lim_{n\rightarrow\infty} a_{nk}=\alpha_k \mbox{~for all}~ k\\
&\exists~ (\alpha_k) \displaystyle\sup_{n}\sup_{k}\Big(\displaystyle|a_{nk}-\alpha_k|
L\Big)^{p_k}<\infty \mbox{~for all}~ L\\
& \exists~ (\alpha_k) \displaystyle\sup_{n}\sum_{k}\Big(\displaystyle|a_{nk}-\alpha_k|
L\Big)^{{p_k}'}<\infty \mbox{~for all}~ L\\
&\displaystyle\sup_{n}\sup_{k}\displaystyle|a_{nk}L^{-1}|^{p_k}<\infty \mbox{~for some}~ L\\
& \displaystyle\sup_{F}\displaystyle\sum_{n}\Big|\displaystyle\sum_{k\in F} a_{nk} L^{\frac{-1}{p_k}}\Big|<\infty \mbox{~for some}~ L\\
& \displaystyle\sum_{n}\Big|\displaystyle\sum_{k} a_{nk}\Big|<\infty\\
& \displaystyle\sup_{F}\displaystyle\sum_n\Big|\displaystyle\sum_{k\in F} a_{nk} L^{\frac{1}{p_k}}\Big|<\infty \mbox{~for all~} L\\
&\displaystyle\sup_{n}\displaystyle\sum_{k} |a_{nk}| L^{-\frac{1}{p_k}}<\infty \mbox{~for some}~ L
\end{align}
\begin{align}
&\displaystyle\sup_{n}\Big|\displaystyle\sum_{k} a_{nk}\Big|<\infty.\\
& \displaystyle\sup_{n}\displaystyle\sum_k|a_{nk}|
L^{\frac{1}{p_k}}<\infty \mbox{~for all}~ L\\
& \exists~ (\alpha_k) \displaystyle\sup_{n}\sum_{k}|a_{nk}-\alpha_k|L^{-\frac{1}{p_k}}<\infty \mbox{~for some}~ L\\
& \exists~ \alpha \displaystyle\lim_n\Big|\displaystyle\sum_{k}a_{nk}-\alpha\Big|=0\\
& \displaystyle\sup_{n} \sum_{k}|a_{nk}|L^{\frac{1}{p_k}}<\infty  \mbox{~for all}~ L\\
& \exists~ (\alpha_k) \displaystyle\lim_n \displaystyle\sum_{k} |a_{nk}-\alpha_k|L^{\frac{1}{p_k}}=0 \mbox{~for all}~ L\\
& \displaystyle\sup_{n}\displaystyle\sum_{k} |a_{nk} L^{-1}|^{p_k'}<\infty \mbox{~for some}~ L
\end{align}
\begin{lem}\cite{GRO}{\label{lem1}}
 $(a)$ if $1 < p_k \leq H < \infty$. Then we have\\
$(i)$ $A\in (l(p), l_1)$ if and only if $(4.5)$ holds.\\
$(ii)$ $A\in (l(p), c_0)$ if and only if $(4.6)$ and $(4.8)$ hold.\\
$(iii)$ $A \in (l(p), c)$ if and only if $(4.10)$, $(4.12)$ and $(4.24)$ hold.\\
$(iv)$ $A \in (l(p), l_\infty)$ if and only if $(4.24)$ holds.\\
$(b)$ if $0< p_k \leq 1$. Then we have\\
$(i)$ $A\in (l(p), l_1)$ if and only if $(4.4)$ holds.\\
$(ii)$ $A\in (l(p), c_0)$ if and only if $(4.6)$ and $(4.7)$ hold.\\
$(iii)$ $A \in (l(p), c)$ if and only if $(4.9)$, $(4.10)$ and $(4.11)$ hold.\\
$(iv)$ $A \in (l(p), l_\infty)$ if and only if $(4.13)$ holds.
\end{lem}

\begin{lem}\cite{GRO}{\label{lem2}}
For $0 < p_k \leq H < \infty$. Then we have\\
$(i)$ $A \in (c_0(p), l_1)$ if and only if $(4.14)$ holds.\\
$(ii)$ $A\in (c(p), l_1)$ if and only if $(4.14)$ and $(4.15)$ hold .\\
$(iii)$ $A\in (l_{\infty}(p), l_1)$ if and only if $(4.16)$ holds.
\end{lem}

\begin{lem}\cite{GRO}{\label{lem3}}
For $0 < p_k \leq H < \infty$. Then we have\\
$(i)$ $A \in (c_0(p), l_{\infty})$ if and only if $(4.17)$ holds.\\
$(ii)$ $A\in (c(p), l_{\infty})$ if and only if $(4.17)$ and $(4.18)$ hold .\\
$(iii)$ $A\in (l_{\infty}(p), l_{\infty})$ if and only if $(4.19)$
holds.
\end{lem}

\begin{lem}\cite{GRO}{\label{lem40}}
For $0 < p_k \leq H < \infty$, we have\\
 $(i)$ $A \in (c_0(p), c)$ if and only if $(4.10)$, $(4.17)$ and $(4.20)$ hold.\\
$(ii)$ $A\in (c(p), c)$ if and only if $(4.10)$, $(4.17)$, $(4.20)$, $(4.21)$ hold.\\
$(iii)$ $A\in (l_{\infty}(p), c)$ if and only if $(4.22)$, $(4.23)$ hold.
\end{lem}

We now define the following sets to obtain the $\alpha$-dual of the spaces $X(r,s,t, p ; B)$:
\begin{align*}
&H_1(p)=\bigcup_{L\in \mathbb{N}}\Big\{a=(a_n)\in w:  \displaystyle\sup_{F}\displaystyle\sum_{n}\bigg|\displaystyle\sum_{k\in F} \displaystyle\sum_{j=k}^{n}(-1)^{j-k}\frac{(-v)^{n-j}}{u^{n-j+1}}\frac{D_{j-k}^{(s)}}{t_j}r_k
a_n L^{\frac{-1}{p_k}}\bigg|<\infty\Big\}
\end{align*}
\begin{align*}
&H_2(p)=\Big\{a=(a_n)\in w:  \displaystyle\sum_{n}\bigg|\displaystyle\sum_{k} \displaystyle\sum_{j=k}^{n}(-1)^{j-k}\frac{(-v)^{n-j}}{u^{n-j+1}}\frac{D_{j-k}^{(s)}}{t_j}r_k
a_n \bigg|<\infty\Big\}\\
&H_3(p)=\bigcap_{L\in \mathbb{N}}\Big\{a=(a_n)\in w:  \displaystyle\sup_{F}\displaystyle\sum_n\bigg|
\displaystyle \sum_{k\in F}\Big(\displaystyle\sum_{j=k}^{n}(-1)^{j-k}\frac{(-v)^{n-j}}{u^{n-j+1}}\frac{D_{j-k}^{(s)}}{t_j}r_k
a_n\Big) L^{\frac{1}{p_k}}\bigg|<\infty\Big\}\\
& H_4(p)=\Big\{a=(a_n)\in w: \displaystyle\sup_{F}\displaystyle\sup_{k}\bigg|\displaystyle\sum_{n\in F}
\displaystyle\sum_{j=k}^{n}(-1)^{j-k}\frac{(-v)^{n-j}}{u^{n-j+1}}\frac{D_{j-k}^{(s)}}{t_j}r_k
a_n\bigg|^{p_k}<\infty \Big\}\\
& H_5(p)=\bigcup_{L\in \mathbb{N}}\Big\{a=(a_n)\in w:
\displaystyle\sup_{F}\displaystyle\sum_{k}\bigg|\displaystyle\sum_{n \in
F}\displaystyle\sum_{j=k}^{n}(-1)^{j-k}\frac{(-v)^{n-j}}{u^{n-j+1}}\frac{D_{j-k}^{(s)}}{t_j}r_k
a_nL^{-1}\bigg|^{p_k'}<\infty\Big\}.
\end{align*}

\begin{thm}
$(a)$If $ p_k > 1$, then $[l(r, s, t, p ; B)]^\alpha= H_5(p)$ and $[l(r, s, t, p ; B)]^\alpha=  H_4(p)$ for $0 < p_k \leq
1$.\\
$(b)$ For $0 < p_k \leq H < \infty$, then \\
$(i)$ $[c_0(r, s, t, p ; B)]^\alpha= H_1(p)$.\\
$(ii)$ $ [c(r, s, t, p ; B)]^\alpha= H_1(p)\cap H_2(p)$.\\
$(iii)$ $[l_\infty(r, s, t, p ; B)]^\alpha= H_3(p)$.
\end{thm}
\begin{proof}$(a)$ Let $p_k>1$ $\forall k$, $a=(a_{n}) \in w$, $x\in l(r, s, t, p; B)$ and $y\in l(p)$. Then for each $n$, we have
$$ a_{n}x_n = \displaystyle\sum_{k=0}^{n}\sum_{j=k}^{n}(-1)^{j-k}\frac{(-v)^{n-j}}{u^{n-j+1}}\frac{D_{j-k}^{(s)}}{t_j}r_k
a_ny_k =(Cy)_{n},$$
where the matrix $C=(c_{nk})_{n, k}$ is defined as
$$
c_{nk} = \left\{
\begin{array}{ll}
 \displaystyle\sum_{j=k}^{n}(-1)^{j-k}\frac{(-v)^{n-j}}{u^{n-j+1}}\frac{D_{j-k}^{(s)}}{t_j}r_k
a_n& \mbox{if}~~   0\leq k \leq n \\
0 & \mbox{if} ~~ k>n,
\end{array}
\right.
$$
and $x_n$ is given in $(4.3)$. Thus for each $x \in l(r, s, t, p;
B)$, $(a_nx_{n})_{n} \in l_{1}$ if and only if $(Cy)_{n} \in
l_{1}$ where $y \in l(p)$. Therefore $a=(a_{n}) \in [l(r, s, t, p;
B)]^{\alpha}$ if and only if $C \in (l(p), l_1)$. By using
Lemma \ref{lem1}$(a)$, we have
$$ [l(r, s, t, p; B)]^{\alpha}= H_5(p).$$
If $0<p_k\leq1$ $\forall k$, then using
Lemma \ref{lem1}$(b)$, we have $[l(r, s, t, p; B)]^{\alpha}= H_4(p)$.\\
$(b)$ In a similar way, using Lemma \ref{lem2}, it can be derived that $[c_0(r, s, t, p ; B)]^\alpha= H_1(p)$, $[c(r, s, t, p ; B)]^\alpha$ $= H_1(p)\cap H_2(p)$ and $[l_\infty(r, s, t, p ; B)]^\alpha= H_3(p)$.
\end{proof}
To find the $\gamma$-dual of the spaces $X(r, s, t, p; B)$ for $X\in\{l_\infty(p), c(p), c_0(p), l(p)\}$, we consider the following sets:
\begin{align*}
&\Gamma_1(p)= \bigcup_{L\in \mathbb{N}}\Big\{ a=(a_k)\in w: \displaystyle\sup_{n}\displaystyle\sum_{k} |e_{nk}| L^{-\frac{1}{p_k}}<\infty\Big\}\\
&\Gamma_2(p)= \Big\{ a=(a_k)\in w: \displaystyle\sup_{n}\Big|\displaystyle\sum_{k} e_{nk}\Big|<\infty\Big\}\\
&\Gamma_3(p)= \bigcap_{L\in \mathbb{N}}\Big\{ a=(a_k)\in w: \displaystyle\sup_{n}\displaystyle\sum_k|e_{nk}| L^{\frac{1}{p_k}}<\infty\Big\}
\end{align*}
\begin{align*}
&\Gamma_4(p)= \bigcup_{L\in \mathbb{N}}\Big\{ a=(a_k)\in w: \displaystyle\sup_{n}\displaystyle\sup_{k}| e_{nk}L^{-1}|^{p_k}<\infty\Big\}\\
&\Gamma_5(p)= \bigcup_{L\in \mathbb{N}}\Big\{ a=(a_k)\in w: \displaystyle\sup_{n}\displaystyle\sum_{k}|e_{nk}L^{-1}|^{p_k'}<\infty\Big\},
\end{align*}
where the matrix $E=(e_{nk})$ is defined as
\begin{equation}{\label{eq20}}
\displaystyle e_{nk}= \left\{
\begin{array}{ll}
   \displaystyle r_k \bigg[\frac{1}{u}\frac{a_k}{s_0t_k}
+
\displaystyle\sum_{j=k}^{k+1}(-1)^{j-k}\frac{D_{j-k}^{(s)}}{t_j}\bigg(\displaystyle\sum_{l=k+1}^{n}\frac{(-v)^{l-j}}{u^{l-j+1}}a_l\bigg)
\\ ~~~~~~~~~~ ~~~~~~~~~+
\displaystyle\sum_{j=k+2}^{n}(-1)^{j-k}\frac{D_{j-k}^{(s)}}{t_j}\bigg(\displaystyle\sum_{l=j}^{n}\frac{(-v)^{l-j}}{u^{l-j+1}}a_l\bigg)\bigg ] & \quad 0\leq k \leq n,\\
    0 & \quad k > n.
\end{array}\right.
\end{equation}
\begin{thm}
$(a)$ If $ p_k > 1$, then $[l(r, s, t, p ; B)]^\gamma=  \Gamma_5(p)$
~and ~$[l(r, s, t, p ; B)]^\gamma=  \Gamma_4(p)$ for $0 < p_k\leq1$.\\
$(b)$ For $0 < p_k \leq H < \infty$ then \\
$(i)$ $[c_0(r, s, t, p ; B)]^\gamma= \Gamma_1(p)$,\\
$(ii)$ $ [c(r, s, t, p ; B)]^\gamma= \Gamma_1(p)\cap \Gamma_2(p)$,\\
$(iii)$ $[l_\infty(r, s, t, p ; B)]^\gamma= \Gamma_3(p)$.
\end{thm}
\begin{proof}
$(a)$ Let $p_k > 1$ for all $k$, $a
=(a_n)\in w$, $x \in l(r,s,t, p ; B)$ and $y \in l(p)$. Then using (4.3), we have
\begin{align*}
\displaystyle\sum_{k=0}^{n}a_kx_k &
=\displaystyle \sum_{k=0}^{n}\sum_{l=0}^{k}\sum_{j=l}^{k}(-1)^{j-l}\frac{(-v)^{k-j}}{u^{k-j+1}}\frac{D_{j-l}^{(s)}}{t_j}r_l
y_la_k\\
&=\displaystyle\sum_{k=0}^{n-1}\displaystyle\sum_{l=0}^{k}\sum_{j=l}^{k}(-1)^{j-l}\frac{(-v)^{k-j}}{u^{k-j+1}}\frac{D_{j-l}^{(s)}}{t_j}a_kr_l
y_l +
\displaystyle\sum_{l=0}^{n}\sum_{j=l}^{n}(-1)^{j-l}\frac{(-v)^{n-j}}{u^{n-j+1}}\frac{D_{j-l}^{(s)}}{t_j}a_nr_l
y_l\\
& =\bigg[\frac{1}{u}\frac{D_{0}^{(s)}}{t_0}a_0 +
\displaystyle\sum_{j=0}^{1}(-1)^{j}\frac{D_{j}^{(s)}}{t_j}\bigg(\displaystyle\sum_{l=1}^{n}\frac{(-v)^{l-j}}{u^{l-j+1}}a_l\bigg)
+
\sum_{j=2}^{n}(-1)^{j}\frac{D_{j}^{(s)}}{t_j}\bigg(\displaystyle\sum_{l=j}^{n}\frac{(-v)^{l-j}}{u^{l-j+1}}a_l\bigg)\bigg]r_0y_0\\
& ~~+\bigg[\frac{1}{u}\frac{D_{0}^{(s)}}{t_1}a_1 +
\displaystyle\sum_{j=1}^{2}(-1)^{j-1}\frac{D_{j-1}^{(s)}}{t_j}\bigg(\displaystyle\sum_{l=2}^{n}\frac{(-v)^{l-j}}{u^{l-j+1}}a_l\bigg)
+
\sum_{j=3}^{n}(-1)^{j-1}\frac{D_{j-1}^{(s)}}{t_j}\bigg(\displaystyle\sum_{l=j}^{n}\frac{(-v)^{l-j}}{u^{l-j+1}}a_l\bigg)\bigg]r_1y_1\\
&~~ + \ldots + \frac{1}{u}\frac{D_{0}^{(s)}}{t_n}a_nr_n y_n\\
&
=\displaystyle\sum_{k=0}^{n}\bigg[\frac{1}{u}\frac{a_k}{s_0t_k}
+
\displaystyle\sum_{j=k}^{k+1}(-1)^{j-k}\frac{D_{j-k}^{(s)}}{t_j}\bigg(\displaystyle\sum_{l=k+1}^{n}\frac{(-v)^{l-j}}{u^{l-j+1}}a_l\bigg)
+
\sum_{j=k+2}^{n}(-1)^{j-k}\frac{D_{j-k}^{(s)}}{t_j}\bigg(\displaystyle\sum_{l=j}^{n}\frac{(-v)^{l-j}}{u^{l-j+1}}a_l\bigg)\bigg]r_ky_k\\
& =(Ey)_{n}
\end{align*}
where $E$ is the matrix defined in (\ref{eq20}).
\\Thus $a \in \big[l(r,s,t, p ; B)\big]^{\gamma}$ if and only if $ax=(a_kx_k)\in bs$ for
$x\in l(r,s,t, p ;B)$ if and only if
$\Big(\displaystyle\sum_{k=0}^{n}a_k x_k \Big)_{n}\in
l_{\infty}$, i.e., $(Ey)_{n} \in l_{\infty}$, for $y\in l(p)$. Hence using Lemma \ref{lem1}$(a)$, we have
$$ \big[l(r,s,t, p ; B)\big]^{\gamma} = \Gamma_{5}(p).$$
If $0<p_k\leq1$ $\forall k$, then using
Lemma \ref{lem1}$(b)$, we have
$[l(r, s, t, p; \Delta)]^{\gamma}= \Gamma_4(p)$.\\
$(b)$ In a similar way, using Lemma \ref{lem3}, we can obtain $[c_0(r, s, t, p ; B)]^\gamma= \Gamma_1(p)$, $ [c(r, s, t, p ; B)]^\gamma= \Gamma_1(p)\cap \Gamma_2(p)$ and
$[l_\infty(r, s, t, p ; B)]^\gamma= \Gamma_3(p)$.
\end{proof}
To obtain $\beta$-duals of $X(r, s, t, p;B)$, we define the following sets:
\begin{align*}
&B_{1}= \Big\{ a=(a_n)\in w: \displaystyle\sum_{l=k+1}^{\infty}\frac{(-v)^{l-j}}{u^{l-j+1}}a_l~~{\rm exists~ for~ all}~ k \Big \}\\
&B_{2}= \Big\{ a=(a_n)\in w: \displaystyle\sum_{j=k+2}^{\infty}(-1)^{j-k}\frac{D_{j-k}^{(s)}}{t_j}\bigg(\displaystyle\sum_{l=j}^{\infty}\frac{(-v)^{l-j}}{u^{l-j+1}}a_l\bigg) ~~{\rm exists~ for~ all}~ k \Big \}\\
&B_{3}= \Big\{ a=(a_n)\in w: \Big(\frac{r_k a_k}{t_k}\Big) \in l_{\infty}(p) \Big \}\\
&B_{4}= \bigcup_{L\in \mathbb{N}}\Big\{ a=(a_n)\in w: \displaystyle\sup_{n}\sum_{k}\Big|e_{nk}L^{-1}\Big|^{p_{k}{'}}<\infty\Big \} \\
&B_{5}= \Big\{a=(a_n)\in w: \displaystyle\sup_{n, k}|e_{nk}|^{p_k}< \infty \Big \}\\
&B_{6}= \Big\{ a=(a_n)\in w: \exists (\alpha_k)~\displaystyle\lim_{n \rightarrow \infty}e_{nk}= \alpha_k ~ \forall ~k \Big \},\\
&B_{7}= \bigcap_{L\in \mathbb{N}}\Big\{ a=(a_n)\in w: \exists (\alpha_k)~ \displaystyle\sup_{n, k}\Big(|e_{nk} - \alpha_k | L \Big)^{p_k}< \infty \Big \} \\
&B_{8}= \bigcap_{L\in \mathbb{N}}\Big\{ a=(a_n)\in w: \exists (\alpha_k)~\displaystyle\sup_{n
}\sum_{k}\Big(|e_{nk} - \alpha_k | L \Big)^{p_k'}< \infty \Big \}\\
&B_{9}= \bigcup_{L\in \mathbb{N}}\Big\{ a=(a_n)\in w: \exists (\alpha_k)~ \displaystyle\sup_{n
}\sum_{k}\Big|e_{nk} - \alpha_k \Big| L ^\frac{-1}{p_k}< \infty
\Big \}\\
&B_{10}= \bigcup_{L\in \mathbb{N}}\Big\{ a=(a_n)\in w:\displaystyle\sup_{n
}\sum_{k}|e_{nk}| L^\frac{-1}{p_k}< \infty \Big \}\\
&B_{11}= \Big\{ a=(a_n)\in w: \exists \alpha~\displaystyle\lim_{n
}\Big|\sum_{k}e_{nk}-\alpha\Big| =0 \Big \}\\
&B_{12}= \bigcap_{L\in \mathbb{N}}\Big\{ a=(a_n)\in w: \displaystyle\sup_{n
}\sum_{k}|e_{nk}|L^\frac{1}{p_k} <\infty\Big \}\\
&B_{13}= \bigcap_{L\in \mathbb{N}}\Big\{ a=(a_n)\in w: \exists (\alpha_k)~\displaystyle\lim_{n
}\sum_{k}|e_{nk}-\alpha_k|L^\frac{1}{p_k}=0\Big \}.
\end{align*}
\begin{thm}
$(a)$ If $p_k>1$ for all $k$, then $[l(r,s, t, p; B)]^{\beta} = B_1 \bigcap B_2 \bigcap B_3
\bigcap B_4 \bigcap B_6\bigcap B_8$ and if $0< p_k \leq 1 $ for all $k$, then $[l(r,s, t, p; B)]^{\beta} = B_1 \bigcap B_2
\bigcap B_3 \bigcap B_5\bigcap B_6 \bigcap B_7$.
\\
$(b)$ Let $p_k>0$ for all $k$. Then\\
$(i)$ $[c_0(r,s, t, p; B)]^{\beta} = B_1 \bigcap B_2 \bigcap B_3
\bigcap B_6 \bigcap B_9\bigcap B_{10}$,\\
$(ii)$ $[c(r,s, t, p; B)]^{\beta}= B_1 \bigcap B_2 \bigcap B_3\bigcap B_6\bigcap B_{9}\bigcap B_{10} \bigcap B_{11}$,\\
$(iii)$ $[l_{\infty}(r,s, t, p; B)]^{\beta}= B_1 \bigcap B_2 \bigcap B_3\bigcap B_{12}\bigcap B_{13}$.
\end{thm}
\begin{proof}
$(a)$ Let $p_k>1$ for all $k$. By Theorem 4.5, we have
$$\sum\limits_{k=0}^{n}a_k x_k = (Ey)_n,$$ where the matrix $E$ is
defined in (\ref{eq20}).
Thus $a \in \big[l(r,s,t,p; B)\big]^{\beta}$
if and only if $ax=(a_kx_k)\in cs$, where $x \in l(r,s,t,p; B)$
if and only if $(Ey)_{n} \in c$, where $y\in l(p)$, i.e., $E \in (l(p),c)$. Hence by Lemma \ref{lem1}$(a)$, we have
\begin{align*}
& \displaystyle\sup_{n\in \mathbb{N}_0}\sum_{k=0}^{\infty}\Big|e_{nk}L^{-1}\Big|^{p_{k}{'}}<\infty ~\mbox{for some} ~L\in \mathbb{N},\\
& \exists (\alpha_k)~\displaystyle\lim_{n \rightarrow \infty}e_{nk}= \alpha_k  ~\mbox{for all} ~k, \\
& \exists (\alpha_k)~\displaystyle\sup_{n\in \mathbb{N}_0}\sum_{k=0}^{\infty}\Big(|e_{nk} - \alpha_k| L \Big)^{p_k'}< \infty ~\mbox{for all} ~L\in \mathbb{N}.
\end{align*}
Therefore $[l(r,s, t, p; B)]^{\beta} = B_1
\bigcap B_2 \bigcap B_3 \bigcap B_4 \bigcap B_6\bigcap B_8.$\\
If $0<p_k\leq1$ $\forall k$, then using Lemma \ref{lem1}$(b)$, we have
\begin{align*}
& \displaystyle\sup_{n, k\in \mathbb{N}_0}|e_{nk}|^{p_k}< \infty,\\
& \exists (\alpha_k)~\displaystyle\lim_{n \rightarrow \infty}e_{nk}= \alpha_k  ~\mbox{for all} ~k, \\
& \exists (\alpha_k)~ \displaystyle\sup_{n, k\in \mathbb{N}_0
}\Big(|e_{nk} - \alpha_k | L \Big)^{p_k}< \infty ~\mbox{for all} ~L\in \mathbb{N}.
\end{align*}
Thus $[l(r,s, t, p; B)]^{\beta} = B_1
\bigcap B_2 \bigcap B_3 \bigcap B_5 \bigcap B_6\bigcap B_7$.\\
$(b)$ In a similar way, using Lemma \ref{lem40}, we can obtain the $\beta$-duals of $c_0(r,s, t, p; B)$, $c(r,s, t, p; B)$ and $l_{\infty}(r,s, t, p; B)$.
\end{proof}

\subsection{Matrix mappings}
\begin{thm}Let $\tilde{E}=(\tilde{e}_{nk})$ be the matrix which is same as the matrix
${E}=({e}_{nk})$ defined in (\ref{eq20}), where $a_{k}$ and $a_{l}$ is replaced by $a_{nk}$ and $a_{nl}$ respectively. \\
$(a)$ Let
$1< p_k\leq H< \infty$ for $k \in \mathbb{N}_0$, then $A \in (l(r,s,
t, p; B), l_{\infty})$ if and only if there exists $L \in \mathbb{N}$ such
that
$$ \displaystyle\sup_{n}\sum_{k}|\tilde{e}_{nk}L^{-1}|^{p_{k}^{'}} < \infty \mbox{~~and~}(a_{nk})_{k}
     \in B_1 \bigcap B_2 \bigcap B_3 \bigcap B_4 \bigcap B_6\bigcap B_8.$$
$(b)$ Let $0< p_k \leq 1 $ for $k \in \mathbb{N}$. Then $ A \in
(l(r,s, t, p; B), l_{\infty})$
  if and only if there exists $L \in \mathbb{N}$ such
that $$ \displaystyle\sup_{n}\sup_{k}|\tilde{e}_{nk}L^{-1}|^{p_{k}} < \infty \mbox{~~and~}(a_{nk})_{k}
   \in B_1 \bigcap B_2 \bigcap B_3 \bigcap
B_5\bigcap B_6 \bigcap B_7.$$
\end{thm}
\begin{proof}
$(a)$ Let $p_k>1$ for all $k$. Since
$(a_{nk})_{k}\in \big[l(r,s,t,p;
B)\big]^{\beta}$ for each fixed $n$, $Ax$ exists for all $x\in l(r,s,t,p;
B) $. Now for each $n$, we have
\begin{align*}
&\sum\limits_{k=0}^{m}a_{nk} x_k \\
&=\sum\limits_{k=0}^{m}\displaystyle r_k
\bigg[\frac{1}{u}\frac{a_{nk}}{s_0t_k}
+
\displaystyle\sum_{j=k}^{k+1}(-1)^{j-k}\frac{D_{j-k}^{(s)}}{t_j}\bigg(\displaystyle\sum_{l=k+1}^{n}\frac{(-v)^{l-j}}{u^{l-j+1}}a_{nl}\bigg)
+
\sum_{j=k+2}^{n}(-1)^{j-k}\frac{D_{j-k}^{(s)}}{t_j}\bigg(\displaystyle\sum_{l=j}^{n}\frac{(-v)^{l-j}}{u^{l-j+1}}a_{nl}\bigg)\bigg]y_k\\
   & = \sum\limits_{k=0}^{m}\tilde{e}_{nk}y_k,
\end{align*}
 Taking
$m\rightarrow\infty$, we have
$$\sum\limits_{k=0}^{\infty}a_{nk}
x_k=\sum\limits_{k=0}^{\infty}\tilde{e}_{nk} y_k  \mbox{~~~for all~}
n\in \mathbb{N_{\rm 0}}.$$ We know that for any $T>0$ and any
complex numbers $a, b$
\begin{equation}
|ab|\leq T(|aT^{-1}|^{p'} + |b|^p)
\end{equation} where $p>1$ and $\frac{1}{p}+ \frac{1}{p'}=1$. Now one can easily
find that
$$\displaystyle \sup_{n}\bigg|\sum\limits_{k}a_{nk} x_k\bigg|\leq
\sup_{n}\sum\limits_{k}\Big|\tilde{e}_{nk} \Big|\Big|y_k\Big|
\leq  T\bigg[\sup_{n}\sum\limits_{k}|\tilde{e}_{nk}T^{-1}|^{{p_k}'} +
\sum\limits_{k}|y_{k}|^{{p_k}}\bigg]<\infty.$$
Conversely, assume that $A\in (l(r,s, t, p; B), l_{\infty})$
and $1< p_k\leq H< \infty$ for all $k$. Then $Ax$
exists for each $x\in l(r,s, t, p; B)$, which implies that
$(a_{nk})_{k}\in [l(r,s, t, p;
B)]^\beta$ for each $n$. Thus
$(a_{nk})_{k}\in B_1 \bigcap B_2 \bigcap B_3
\bigcap B_4 \bigcap B_6\bigcap B_8$. Since
$\sum\limits_{k=0}^{\infty}a_{nk}
x_k=\sum\limits_{k=0}^{\infty}\tilde{e}_{nk} y_k $, we have
$\tilde{E}=(\tilde{e}_{nk})\in (l(p), l_\infty)$. Now using Lemma \ref{lem1}$(a)$, we have
$\displaystyle\sup_{n}\sum_{k}\Big|\tilde{e}_{nk}L^{-1}\Big|^{p_{k}^{'}}
<\infty$ for some
$L\in \mathbb{N}$. This completes the proof.\\
$(b)$  We omit the proof of this part as it is similar to the previous part.
\end{proof}

\begin{thm}
$(a)$ Let $1< p_k\leq H< \infty$ for $k \in \mathbb{N}$, then $A \in
(l(r,s, t, p; B), l_{1})$ iff there exists $L\in \mathbb{N}$
such that
\begin{center}
$ \displaystyle\sup_{{F}} \sum_{k}\Big| \sum_{ n \in
F}\tilde{e}_{nk}L^{-1}\Big|^{p^{'}_{k}} < \infty ~~ {\rm for ~ some~ L \in
\mathbb{N}}
 \mbox{~~and~} (a_{nk})_{k}  \in B_1 \bigcap B_2 \bigcap B_3 \bigcap B_4\bigcap B_6\bigcap B_8.$
 \end{center}
$(b)$ Let $0< p_k \leq 1 $ for $k \in \mathbb{N}$. Then $ A \in
(l(r,s, t, p; B), l_{1})$
  iff
\begin{center}
   $ \displaystyle\sup_{ F} \sup_{k}\Big| \sum_{ n \in F}\tilde{e}_{nk}\Big|^{p_{k}} <
  \infty$
and $ ( a_{nk})_{k } \in B_1 \bigcap B_2
\bigcap B_3 \bigcap B_5\bigcap B_6 \bigcap B_7.$
\end{center}
\end{thm}
\begin{proof}
We omit the proof as it follows as the same way.
\end{proof}

\section{Measure of noncompactness and compact operators on the space $l_p(r, s,t; B)$}
In this section, we concentrate on $l_p(r, s,t; B)$, $p\geq 1$, which is a $BK$ space and establish some identities or estimates for the Hausdorff measure of noncompactness of certain matrix operators on the space $l_p(r, s,t; B)$. Moreover, we characterize some classes of compact
operators on this space.\\
The Hausdorff measure of noncompactness was first introduced and studied by Goldenstein, Gohberg and Markus in 1957 and
later on stuided by Istr\v{a}\c{t}esku in 1972 \cite{IST}.
 It is quite natural to find necessary and sufficient conditions for a matrix mapping between $BK$ spaces to define a compact operator as
 the matrix transformations between $BK$ spaces are continuous. This can be achieved with the help of Hausdorff measure of noncompactness.
 Recently several authors, namely, Malkowsky and Rako\v{c}evi\'{c} \cite{MAL5}, Dojolovi\'{c} et al. \cite{DJO2}, Dojolovi\'{c} \cite{DJO},  Mursaleen and Noman (\cite{MUR2}, \cite{MUR3}), Ba\c{s}arir and Kara \cite{BASA1}
 etc. have established some identities or estimates for the operator norms and the Hausdorff measure of noncompactness of matrix operators from an
 arbitrary $BK$ space to arbitrary $BK$ space. Let us recall some definitions and well-known results.

 Let $X$, $Y$ be two Banach spaces and $S_X$ denotes the unit sphere in $X$, i.e., $S_X=\{x\in X: \|x\|=1\}$.
 We denote by $\mathcal{B}(X, Y)$, the set of all bounded (continuous) linear operators $L: X\rightarrow Y$, which is a Banach space with the operator norm $\|L\|=\displaystyle\sup_{x\in S_X}\|L(x)\|_Y$ for all $L\in \mathcal{B}(X, Y)$. A linear operator $L: X\rightarrow Y$ is said to be compact if the domain of $L$ is all of $X$ and for every bounded sequence $(x_n)\in X$, the sequence $(L(x_n))$ has a subsequence which is convergent in $Y$ and we denote by $\mathcal{C}(X, Y)$, the class of all compact operators in $\mathcal{B}(X, Y)$. An operator $L\in \mathcal{B}(X, Y)$ is said to be finite rank if ${\rm dim}R(L)<\infty$, where $R(L)$ is the range space of $L$.
If $X$ is a $BK$ space and $a=(a_k)\in w$, then we consider
\begin{equation}{\label{eq0}}
\|a\|_X^*=\displaystyle\sup_{x\in S_X}\Big|\displaystyle\sum_{k=0}^{\infty}a_kx_k\Big|,
\end{equation}
provided the expression on the right side exists and is finite which is the case whenever $a\in X^\beta$ \cite{MUR3}.\\
Let $(X,d )$ be a metric space and $\mathcal{M}_{X}$ be the class of all bounded subsets of $X$.
Let $B(x, r) = \{y \in X : d(x,y) < r \}$ denotes the open ball of radius $r> 0$ with centre at $x$.
The Hausdorff measure of noncompactness of a set $Q \in \mathcal{M}_{X}$, denoted by $\chi(Q)$, is defined as
$$ \displaystyle\chi(Q) = \inf \Big \{ \epsilon > 0: Q \subset \bigcup_{i=0}^{n}B(x_i, r_i), x_i \in X, r_i < \epsilon , n \in \mathbb{N}_0\Big\}.$$
The function $\chi: \mathcal{M}_{X} \rightarrow [0, \infty) $ is called the Hausdorff measure of noncompactness. The basic properties of the Hausdorff measure of noncompactness can be found in (\cite{MAL5}, \cite{DJO2}, \cite{MAL3}, \cite{MAL4}).
For example, if $Q, Q_1$ and $Q_2$ are bounded subsets of a metric space $(X,d)$ then
\begin{align*}
&\chi(Q) =0 \mbox{~if and only if~} Q \mbox{~is totally bounded} \mbox{~~and}\\
&\mbox{if~} Q_1 \subset Q_2 \mbox{~then~} \chi(Q_1) \leq \chi(Q_2).
\end{align*}
Also if $X$ is a normed space, the function $\chi$ has some additional properties due to linear structure, namely,
\begin{align*}
&\chi(Q_1 + Q_2) \leq  \chi(Q_1) + \chi( Q_2),\\
& \chi( \alpha Q) = |\alpha| \chi(Q) ~ \mbox{for all }~ \alpha \in \mathbb{K}.
\end{align*}
Let $\phi$ denotes the set of all finite sequences, i.e., of sequences that terminate in zeros. Throughout we denote $p'$ as the conjugate of $p$ for $1\leq p<\infty$, i.e., $p{'} =\frac{ p }{p-1}$ for $p>1$ and $p'=\infty$ for $p=1$. The following known results are fundamental for our investigation.
\begin{lem}{\rm \cite{MUR3}}{\label{lem4}}
Let $1 \leq p< \infty$ and $A\in (l_p, c)$. Then the followings hold:
\begin{align*}
&(i)~\alpha_k= \displaystyle\lim_{n\rightarrow\infty}{a}_{nk} \mbox{~exists for all~} k\in \mathbb{N}_{0},\\
&(ii)~\alpha=(\alpha_k)\in l_{p'},\\
&(iii)~\displaystyle\sup_{n}\|A_n-\alpha\|_{l_{p'}}<\infty, \\
&(iv)~\displaystyle\lim_{n\rightarrow\infty}A_n(x)=\displaystyle\sum_{k=0}^{\infty}\alpha_k x_k \mbox{~for all~}x=(x_k)\in l_p.
\end{align*}
\end{lem}
\begin{lem}{(\rm \cite{MAL5}, Theorem 1.29)}{\label{lem5}}
Let $1 \leq p < \infty$. Then we have $l_{p}^{\beta} = l_{p{'}}$ and $\|a\|_{l_p}^*= \|a\|_{l_{p{'}}}$ for all $a\in l_{p{'}}$.
\end{lem}
\begin{lem}{\rm \cite{MUR3}}{\label{lem6}}
Let $X \supset \phi$ and $Y$ be $BK$ spaces. Then we have $(X, Y)\subset \mathcal{B}(X, Y)$, i.e., every matrix $A\in (X, Y)$ defines an operator $L_A\in \mathcal{B}(X, Y)$, where $L_A(x)=Ax$ for all $x\in X$.
\end{lem}

\begin{lem}{\rm \cite{DJO}}{\label{lem7}}
Let $X\supset\phi$ be a $BK$ space and $Y$ be any of the spaces $c_0$, $c$ or $l_\infty$. If $A\in (X, Y)$, then we have
$$\|L_A\|= \|A\|_{(X, l_\infty)}=\displaystyle\sup_{n}\|A_n\|_X ^{*}<\infty.$$
\end{lem}

\begin{lem}{\rm \cite{MAL5}}{\label{lem9}}
Let $Q \in \mathcal{M}_{X}$, where $X = l_p$ for $1 \leq p < \infty$ or $c_0$. If $P_m : c_0 \rightarrow c_0$ $(m \in \mathbb{N}_{0})$ be the operator defined by $ P_m(x) = (x_0, x_1, \cdots, x_m, 0, 0, \cdots)$ for all $x =(x_k) \in X$. Then we have
$$ \chi (Q) = \displaystyle \lim_{m \rightarrow \infty}\Big( \sup_{x \in Q}\|(I -P_m)(x)\| \Big),$$
where $I$ is the identity operator on $X$.
\end{lem}
Let $z=(z_n) \in c$. Then $z$ has a unique representation $z = \ell e + \displaystyle \sum_{n=0}^{\infty}(z_n -\ell)e_{n}$, where $\ell = \displaystyle \lim_{n \rightarrow \infty} z_n$. We now define the projections
$P_m$ $(m \in \mathbb{N}_{0})$ from $c$ onto the linear span of $\{e, e_0, e_1, \cdots, e_m \}$ as
$$P_{m}(z) = \ell e + \displaystyle \sum_{n=0}^{m}(z_n -\ell)e_{n},$$
for all $z \in c$ and $\ell= \displaystyle \lim_{n \rightarrow \infty} z_n$.\\
 Then the following result gives an estimate for the Hausdorff measure of noncompactness in the $BK$ space $c$.
\begin{lem}{\rm  \cite{MAL5}}{\label{lem10}}
Let $Q\in \mathcal{M}_c$ and $P_m: c\rightarrow c$ be the projector from $c$ onto the linear span of $\{e, e_{0}, e_{1}, \ldots e_m\}$. Then we have
$$\frac{1}{2} \displaystyle\lim_{m\rightarrow\infty}\Big( \displaystyle\sup_{x\in Q}\|(I-P_m)(x)\|_{l_{\infty}}\Big)\leq \chi(Q)\leq \displaystyle\lim_{m\rightarrow\infty}\Big( \displaystyle\sup_{x\in Q}\|(I-P_m)(x)\|_{l_{\infty}}\Big),$$
where $I$ is the identity operator on $c$.
\end{lem}

\begin{lem}{\rm  \cite{MAL5}}{\label{lem8}}
Let $X,Y$ be two Banach spaces and $L \in \mathcal{B}(X, Y)$. Then $$\|L\|_{\chi} = \chi(L(S_X))$$
and $$L \in \mathcal{C}(X, Y) ~\mbox{if and only if}~ \|L\|_{\chi} =0.$$
\end{lem}

Let $\mathcal{F}_{m} = \{F \in \mathcal{F}: n > m, ~\forall n \in F \}$, $m \in \mathbb{N}$ and $\mathcal{F}$ is the collection of nonempty and finite subsets of $\mathbb{N}$.
\begin{lem}{\cite{MUR2}}{\label{lem80}}
Let $X \supset \phi$ be a $BK$ space. \\
$(a)$ If $A \in (X, c_0)$, then
$$ \|L_A \|_{\chi} = \displaystyle \limsup_{n \rightarrow \infty} \| A_n \|_{X}^{*} $$
and
\begin{center}
$L_A$ is compact if and only if $\displaystyle\lim_{n \rightarrow \infty} \| A_n  \|_{X}^{*} =0$.
\end{center}
$(b)$ If $A \in (X, l_{\infty})$, then
$$0 \leq \|L_A \|_{\chi} \leq \displaystyle \limsup_{n \rightarrow \infty} \| A_n \|_{X}^{*} $$ and
\begin{center}
$L_A$ is compact if and only if $\displaystyle\lim_{n \rightarrow \infty} \| A_n \|_{X}^{*} =0$.
\end{center}
$(c)$ If $A \in (X, l_1)$, then
$$ \displaystyle\lim_{m \rightarrow \infty} \Big( \sup_{F \in \mathcal{F}_{m}} \Big \|\sum_{n \in F} A_n \Big \|_{X}^{*} \Big) \leq \|L_A \|_{\chi} \leq 4 \displaystyle\lim_{m \rightarrow \infty} \Big( \sup_{F \in \mathcal{F}_{m}} \Big \|\sum_{n \in F} A_n \Big \|_{X}^{*} \Big)$$
and
\begin{center}
$L_A$ is compact if and only if $\displaystyle\lim_{m \rightarrow \infty} \Big( \sup_{F \in \mathcal{F}_{m}} \Big \|\sum_{n \in F} A_n \Big \|_{X}^{*} \Big) =0$.
\end{center}
\end{lem}

We establish the following lemmas which are required to characterize the classes of compact operators with the help of Hausdorff measure of noncompactness.
\begin{lem}{\label{lem11}}
If $ a=(a_k)\in [l_p(r, s, t; B)]^ \beta$, then
$\tilde{a}= (\tilde{a}_k)\in l_{p'}$ and the equality $$\displaystyle\sum_{k=0}^{\infty}a_k x_k=\displaystyle\sum_{k=0}^{\infty}\tilde{a}_k y_k$$ holds for every $x=(x_k)\in l_{p}(r, s, t; B)$ and $y =(y_k) \in l_p$, where $y = (A(r,s,t). B)x$. In addition
\begin{equation}{\label{eq1}}
 \tilde{a}_{k}= r_{k} \bigg [ \frac{a_{k}}{s_{0}t_{k}}\frac{1}{u} + \sum_{i=k}^{k+1}(-1)^{i-k} \frac{D_{i-k}^{(s)}}{t_{i}}\Big( \sum_{j=k+1}^{\infty}\frac{(-v)^{j-i}}{u^{j-i+1}} a_j \Big)  + \sum_{i=k+2}^{\infty}(-1)^{i-k} \frac{D_{i-k}^{(s)}}{t_{i}}\Big( \sum_{j=i}^{\infty}\frac{(-v)^{j-i}}{u^{j-i+1}} a_j \Big)
  \bigg ].
\end{equation}
\end{lem}
\begin{proof}
Let $a=(a_k)\in [l_p(r, s, t; B)]^ \beta$. Then by (\cite{DJO}, Theorem 2.3, Remark 2.4), we have $R(a)= (R_{k}(a)) \in l_{p}^{\beta} =l_{p'}$ and also
$$\displaystyle\sum_{k=0}^{\infty}a_k x_k=\displaystyle\sum_{k=0}^{\infty}R_{k}(a) T_k(x) \quad \forall~x \in l_p(r, s, t; B),$$
where
\begin{center}
$R_{k}(a) = \displaystyle r_{k} \bigg [ \frac{a_{k}}{s_{0}t_{k}}\frac{1}{u} + \sum_{i=k}^{k+1}(-1)^{i-k} \frac{D_{i-k}^{(s)}}{t_{i}}\Big( \sum_{j=k+1}^{\infty}\frac{(-v)^{j-i}}{u^{j-i+1}} a_j \Big)  + \sum_{i=k+2}^{\infty}(-1)^{i-k} \frac{D_{i-k}^{(s)}}{t_{i}}\Big( \sum_{j=i}^{\infty}\frac{(-v)^{j-i}}{u^{j-i+1}} a_j \Big)
  \bigg ]= \tilde{a}_k.$
\end{center}
and $ y =T(x)= (A(r,s,t). B)x$. This completes the proof.
\end{proof}

\begin{lem}{\label{lem12}}
Let $1 \leq p < \infty$. Then we have
$$\|a\|_{l_{p}(r, s, t; B)}^*= \|\tilde{a}\|_{l_{p'}}=  \left\{
\begin{array}{ll}
    \displaystyle \Big(\sum_{k=0}^{\infty}|\tilde{a_k}|^{p{'}} \Big)^{1 \over p'} & \quad 1< p< \infty \\
    \displaystyle\sup_{k}|\tilde{a}_{k}| & \quad p= 1.
\end{array}\right.$$
for all $a=(a_k)\in [l_{p}(r, s, t; B)]^\beta$, where $\tilde{a}=(\tilde{a}_{k})$ is defined in (\ref{eq1}).
\end{lem}
\begin{proof}
Let $a=(a_k)\in [l_p(r, s, t; B)]^\beta$. Then from Lemma \ref{lem11}, we have $\tilde{a}=(\tilde{a}_{k})\in l_{p'}$.
Also $x\in S_{l_p(r, s, t; B)}$ if and only if $y=T(x) \in S_{l_{p}}$ as $\|x\|_{l_p(r, s, t; B)}= \|y\|_{l_p}$. From (\ref{eq0}), we have
$$\|a\|_{l_p(r, s, t; B)}^*= \displaystyle\sup_{x\in S_{l_p(r, s, t; B)}}\Big|\displaystyle\sum_{k=0}^{\infty}a_k x_k\Big|=\sup_{y\in S_{l_p}}\Big|\displaystyle\sum_{k=0}^{\infty}\tilde{a}_{k} y_k\Big|=  \|\tilde{a}\|_{l_p}^*.$$
Using Lemma \ref{lem5}, we have $\|a\|_{l_p(r, s, t; B)}^*=\|\tilde{a}\|_{l_p}^*= \|\tilde{a}\|_{l_{p'}}$, which is finite as
$\tilde{a}\in l_{p'}$. This completes the proof.
\end{proof}

\begin{lem}{\label{lem13}}
Let $1 \leq p< \infty$, $Y$ be any sequence space and $A=(a_{nk})_{n,k}$ be an infinite matrix. If $A\in (l_p(r, s, t; B), Y)$ then $\tilde{A} \in (l_p, Y)$ such that $Ax= \tilde{A}y$ holds for all $x\in l_p(r, s, t; B)$ and $y\in l_p$, which are connected by the relation $y = (A(r,s,t). B)x$ and $\tilde{A}=(\tilde{a}_{nk})_{n,k}$ is given by
$$\tilde{a}_{nk}= r_{k} \bigg [ \frac{a_{nk}}{s_{0}t_{k}}\frac{1}{u} + \sum_{i=k}^{k+1}(-1)^{i-k} \frac{D_{i-k}^{(s)}}{t_{i}}\Big( \sum_{j=k+1}^{\infty}\frac{(-v)^{j-i}}{u^{j-i+1}} a_{nj} \Big)  + \sum_{i=k+2}^{\infty}(-1)^{i-k} \frac{D_{i-k}^{(s)}}{t_{i}}\Big( \sum_{j=i}^{\infty}\frac{(-v)^{j-i}}{u^{j-i+1}} a_{nj} \Big)
  \bigg ],$$ provided the series on the right side converges for all $n, k$.
\end{lem}
\begin{proof}
We assume that $A\in (l_p(r, s, t; B), Y)$, then $A_n \in [l_p(r, s, t; B)]^{\beta}$ for all $n$. Thus it follows from Lemma \ref{lem11}, we have $\tilde{A}_n \in l_{p}^{\beta}= l_{p'}$ for all $n$ and $Ax = \tilde{A}y$ holds for every $x \in l_p(r, s, t; B)$ and $y \in l_p$, which are connected by the relation $y = (A(r, s, t). B)x$. Hence $\tilde{A}y \in Y$. Since $x = B^{-1}.A(r, s, t)^{-1}y$, for every $y \in l_p$, we get some $x \in l_p(r, s, t; B)$ and hence $\tilde{A} \in (l_p, Y)$. This completes the proof.
\end{proof}

\begin{thm}
Let $1< p< \infty$.\\
$(a)$ If $A\in (l_p(r, s, t; B), c_0)$ then
\begin{equation}{\label{eq2}}
 \|L_A\|_\chi= \displaystyle\limsup_{n\rightarrow\infty} \Big(\displaystyle\sum_{k=0}^{\infty}|\tilde{a}_{nk}|^{p'} \Big)^{1 \over p'}
\end{equation}
and ~~~~~~~~~~~~~~~~~$L_A$ is compact if and only if
$\displaystyle \lim_{n \rightarrow \infty}\Big(\displaystyle\sum_{k=0}^{\infty}|\tilde{a}_{nk}|^{p'} \Big)^{1 \over p'} =0$.\\
$(b)$ If $A \in (l_p(r, s, t; B), l_{\infty})$ then
\begin{equation}{\label{eq4}}
 0 \leq \|L_A \|_{\chi} \leq \displaystyle \limsup_{n \rightarrow \infty} \sum_{k=0}^{\infty}|\tilde{a}_{nk} |
\end{equation}
and ~~~~~~~~~~~~~~~~~$L_A$ is compact if and only if
$\displaystyle \lim_{n \rightarrow \infty}\Big(\displaystyle\sum_{k=0}^{\infty}|\tilde{a}_{nk}|^{p'} \Big)^{1 \over p'} =0$.
\end{thm}
\begin{proof}
Let $1< p< \infty$ and $A\in (l_p(r, s, t; B), c_0)$, then $A_n \in [l_p(r, s, t; B)]^{\beta}$ for all $n$ and hence $\tilde{A}_n\in l_{p'}$ by  Lemma \ref{lem11}. Again using Lemma \ref{lem12}, we have
$$\|A_n\|_{l_p(r, s, t; B)}^*= \|\tilde{A_n}\|_{l_{p'}}= \Big(\displaystyle\sum_{k=0}^{\infty}|\tilde{a}_{nk}|^{p'} \Big)^{1 \over p'}.$$
Now by Lemma \ref{lem80}, we have $\|L_A \|_{\chi}=\displaystyle\limsup_{n \rightarrow \infty} \|A_n\|_{l_p(r, s, t; B)}^* =\displaystyle\limsup_{n \rightarrow \infty} \Big(\displaystyle\sum_{k=0}^{\infty}|\tilde{a}_{nk}|^{p'} \Big)^{1 \over p'}$.\\
Using Lemma \ref{lem8}, we have $L$ is compact if and only if $\displaystyle\lim_{n \rightarrow \infty} \Big(\displaystyle\sum_{k=0}^{\infty}|\tilde{a}_{nk}|^{p'} \Big)^{1 \over p'}=0$.\\
Similarly, we can prove the part(b).
\end{proof}

\begin{thm}
If $A\in (l_p(r, s, t; B), c)$ then
\begin{equation}{\label{eq3}}
\frac{1}{2} \displaystyle\limsup_{n\rightarrow\infty}\displaystyle\sum_{k=0}^{\infty}|\tilde{a}_{nk}-\tilde{\alpha}_k| \leq \displaystyle \|L_A\|_\chi \leq \displaystyle\limsup_{n\rightarrow\infty}\displaystyle\sum_{k=0}^{\infty}|\tilde{a}_{nk}-\tilde{\alpha}_k|,
\end{equation}
where $\tilde{\alpha}_k= \displaystyle\lim_{n\rightarrow\infty}\tilde{a}_{nk}$ for all $k$.
\end{thm}
\begin{proof}
Let $A\in (l_p(r, s, t; B), c)$. Then by using Lemma \ref{lem13} \& \ref{lem4}, we can deduce that the expression in (\ref{eq3}) exists.
We write $S= S_{l_p(r, s, t; B)}$ in short. Then by Lemma \ref{lem8}, we have $ \|L_A \|_{\chi} = \chi(AS)$. Since $l_p(r, s, t; B)$ and $c$ are BK spaces,
$A $ induces continuous map $L_{A}$ from $l_p(r, s, t; B)$ to $c$ by Lemma \ref{lem6}. Thus $AS$ is bounded in $c$,
i.e., $AS \in \mathcal{M}_{c}$. Let $P_m : c \rightarrow c$, $(m \in \mathbb{N}_{0})$ be the projection from $c$ onto the
span of $\{e, e_0, e_1, \cdots, e_m \}$ defined by
$$ P_m(z) = \ell e + \sum_{k=0}^{m} (z_k -\ell)e_k,$$
where $\ell = \displaystyle\lim_{k \rightarrow \infty} z_k$.
Thus for every $m$, we have $$ (I -P_m)(z) = \sum_{k=m+1}^{\infty} (z_k - \ell)e_k,$$ where $I$ is the identity operator.
Therefore
$ \|(I -P_m)(z)\|_{\infty} = \displaystyle \sup_{k =m+1}|z_k - \ell| $ for all $z=(z_k) \in c$.
So by applying Lemma \ref{lem10}, we have
\begin{equation}{\label{eq5}}
\frac{1}{2}\displaystyle\lim_{m\rightarrow\infty}\Big( \displaystyle\sup_{x\in S}\|(I-P_m)(Ax)\|_{l_{\infty}}\Big)\leq \| L_A\|_{\chi} \leq \displaystyle\lim_{m\rightarrow\infty}\Big( \displaystyle\sup_{x\in S}\|(I-P_m)(Ax)\|_{l_{\infty}}\Big).
\end{equation}
Since $A \in (l_p(r, s, t; B),  c)$, we have by Lemma \ref{lem13}, $\tilde{A }\in (l_p, c)$ and $Ax = \tilde{A }y$ for every $x \in l_p(r, s, t;B)$ and $y \in l_p$ which are connected by the relation $y =(A(r,s,t).B)x$.
 Again applying Lemma \ref{lem4}, we have $\tilde{\alpha}_k = \displaystyle \lim_{n \rightarrow \infty}\tilde{a}_{nk}$ exists for all $k$,
$ \tilde{\alpha} = (\tilde{\alpha}_k ) \in X ^{\beta} = l_{p{'}}$ and $\displaystyle \lim_{n \rightarrow \infty}\tilde{A}_{n}(y) = \sum_{k=0}^{\infty} \tilde{\alpha}_k y_k.$
Since $ \|(I -P_m)(z)\|_{l_{\infty}} = \displaystyle \sup_{k > m}|z_k - \ell| $, we have
\begin{align*}
\|(I-P_m)(Ax)\|_{\infty} & = \|(I-P_m)(\tilde{A}y)\|_{\infty}\\
& = \sup_{ n> m}\Big | \tilde{A}_n(y) - \sum_{k=0}^{\infty} \tilde{\alpha}_k y_k \Big |\\
& =  \sup_{ n> m}\Big | \sum_{k=0}^{\infty}(\tilde{a}_{nk} -  \tilde{\alpha}_k) y_k \Big |.
\end{align*}
Also we know that $x \in S= S_{l_{p}(r, s, t; B)}$ if and only if $y \in S_{l_p}$. From (\ref{eq0}) and Lemma \ref{lem5}, we deduce that

\begin{align*}
\sup_{x \in S} \|(I-P_m)(Ax)\|_{\infty} & =  \sup_{ n> m} \Big(\sup_{y \in S_{l_{p}}} \Big | \sum_{k=0}^{\infty}(\tilde{a}_{nk} -  \tilde{\alpha}_k) y_k \Big | \Big)\\
& = \sup_{ n> m}\| \tilde{A}_n -\tilde{\alpha }\|_{l_p}^{*}\\
& =  \sup_{ n> m}\| \tilde{A}_n -\tilde{\alpha }\|_{l_{p{'}}}
\end{align*}
Hence from (\ref{eq5}), we have
\begin{center}
 $\frac{1}{2} \displaystyle\limsup_{n\rightarrow\infty}\Big(\displaystyle\sum_{k=0}^{\infty}|\tilde{a}_{nk}-\tilde{\alpha}_k|^{p'} \Big)^{1 \over p'} \leq \displaystyle \|L_A\|_\chi \leq  \displaystyle\limsup_{n\rightarrow\infty}\Big(\displaystyle\sum_{k=0}^{\infty}|\tilde{a}_{nk}-\tilde{\alpha}_k|^{p'}\Big)^{1 \over p'}$.
\end{center}
This completes the proof.
\end{proof}

\begin{thm}
Let $1\leq p < \infty$. If $A \in (l_1(r,s,t; B), l_p)$, then
$$ \|L_A\|_\chi = \displaystyle\lim_{m\rightarrow\infty}\Big(\sup_k \Big(\displaystyle\sum_{n=m+1}^{\infty}|\tilde{a}_{nk}|^{p}\Big)^{1 \over p} \Big). $$
\end{thm}
\begin{proof}
Let $A \in (l_1(r,s,t; B), l_p)$. We write $S= S_{l_1(r, s, t; B)}$ in short. Since $l_1(r, s, t; B)$ and $l_p$ are BK spaces, by Lemma \ref{lem6}, we have $AS \in \mathcal{M}_{l_p}$. Also by Lemma \ref{lem8}, we have $ \|L_A \|_{\chi} = \chi(AS) $.
Now by Lemma \ref{lem9}
$$ \chi(AS) = \displaystyle \lim_{ m \rightarrow \infty}\Big( \sup_{x \in S}\|(I -P_m)(Ax)\|_{l_p} \Big), $$
where $P_m : l_p \rightarrow l_p$ is the the operator defined by $ P_m(x) = (x_0, x_1, \cdots, x_m, 0, 0, \cdots)$ for all $x =(x_k) \in l_p$ and $m \in \mathbb{N}_{0}$. Since $A \in (l_1(r, s, t; B),  l_p)$, we have by Lemma \ref{lem13}, $\tilde{A }\in (l_1, l_p)$ and $Ax = \tilde{A }y$ holds for every $x \in l_1(r, s, t;B)$ and $y \in l_p$ which are connected by the relation $y =(A(r,s,t).B)x$.
Now for each $m \in \mathbb{N}_{0}$,
\begin{align*}
\|(I-P_m)(Ax)\|_{l_p}  &= \|(I-P_m)(\tilde{A}y)\|_{l_p}\\
 & = \displaystyle \Big(\sum_{n= m+1}^{\infty}\Big | \tilde{A}_n(y)\Big |^p \Big)^{1 \over p}\\
& = \Big( \sum_{n= m+1}^{\infty}\Big | \sum_{k=0}^{\infty}\tilde{a}_{nk} y_k \Big |^p  \Big)^{1 \over p}.\\
& \leq  \sum_{k=0}^{\infty}\Big( \sum_{n= m+1}^{\infty}\Big |\tilde{a}_{nk} y_k \Big |^p \Big)^{1 \over p}.\\
& \leq \|y\|_{l_1} \Big(\sup_{k}\Big( \sum_{n= m+1}^{\infty}\Big |\tilde{a}_{nk}  \Big |^p \Big)^{1 \over p} \Big)\\
& = \|x \|_{l_1(r, s, t; B)}\Big(\sup_{k}\Big( \sum_{n= m+1}^{\infty}\Big |\tilde{a}_{nk}  \Big |^p \Big)^{1 \over p} \Big)
\end{align*}
Thus $$\sup_{x \in S} \|(I-P_m)(Ax)\|_{l_p} \leq \sup_{k}\Big( \sum_{n= m+1}^{\infty}\Big |\tilde{a}_{nk}  \Big |^p \Big)^{1 \over p}.$$
Hence $$ \|L_A\|_\chi \leq  \lim_{m \rightarrow \infty} \Big(\sup_{k}\Big( \sum_{n= m+1}^{\infty}\Big |\tilde{a}_{nk}  \Big |^p \Big)^{1 \over p}\Big).$$

Conversely, let $b^{(j)} \in l_1(r, s, t; B)$ such that $(A(r,s,t).B)b^{(j)} = e_{j}$, where the sequence $(b^{(j)})_{j=0}^{\infty}$
is a basis in $l_1(r, s, t; B)$ defined in Theorem 4.3. Thus $Ab^{(j)} = \tilde{A}e_j $ for each $j$. Let $D  = \{ b^{(j)}: j \in \mathbb{N}_{0} \}$. Then $AD \subset AS$ and by the property of $\chi$, we have $\chi(AD) \leq \chi(AS)= \|L_A \|_{\chi}$.
Further, by Lemma 5.5
$$\chi(AD) =  \lim_{m \rightarrow \infty} \Big(\sup_{k}\Big( \sum_{ n=m+1}\Big |\tilde{a}_{nk}  \Big |^p \Big)^{1 \over p}\Big) \leq \|L_A\|_\chi.$$
Thus we have $\|L_A\|_\chi = \displaystyle\lim_{m \rightarrow \infty} \Big(\sup_{k}\Big( \sum_{ n= m+1 }^{\infty}\Big |\tilde{a}_{nk}  \Big |^p \Big)^{1 \over p}\Big).$
\end{proof}

\begin{thm}{\label{thm12}}
Let $ 1< p< \infty$. If $A \in (l_p(r,s,t; B), l_1)$, then
$$\displaystyle\lim_{m \rightarrow \infty} \Big( \sup_{F \in \mathcal{F}_{m}}\Big( \sum_{k=0}^{\infty}\Big |\sum_{n \in F} \tilde{a}_{nk} \Big|^{p{'}}\Big)^{1 \over p{'}} \Big) \leq \|L_A \|_{\chi} \leq 4 \lim_{m \rightarrow \infty} \Big( \sup_{F \in \mathcal{F}_{m}}\Big( \sum_{k=0}^{\infty}\Big |\sum_{n \in F} \tilde{a}_{nk} \Big|^{p{'}}\Big)^{1 \over p{'}} \Big) $$
and
\begin{center}
$L_A$ is compact if and only if $\displaystyle\lim_{m \rightarrow \infty} \Big( \sup_{F \in \mathcal{F}_{m}}\Big( \sum_{k=0}^{\infty}\Big |\sum_{n \in F} \tilde{a}_{nk} \Big|^{p{'}}\Big)^{1 \over p{'}} \Big) =0$.
\end{center}
\end{thm}
\begin{proof}
Let $A \in (l_p(r,s,t; B), l_1)$. Then $A_n \in  [(l_p(r,s,t; B)]^{\beta}$ for all $n$ and hence $\tilde{A}_n\in l_{p'}$ by Lemma \ref{lem11}. Using Lemma \ref{lem12}, we have
$$ \Big \| \sum_{n \in F}A_n \Big \|^{*}_{l_p(r,s,t; B)} =  \Big \| \sum_{n \in F}\tilde{A}_n \Big \|_{l_p'} = \Big( \sum_{k=0}^{\infty}\Big |\sum_{n \in F} \tilde{a}_{nk}\Big|^{p{'}}\Big)^{1 \over p{'}} .$$
Now using Lemma \ref{lem80}(c), we obtain
$$\displaystyle\lim_{m \rightarrow \infty} \Big( \sup_{F \in \mathcal{F}_{m}}\Big( \sum_{k=0}^{\infty}\Big |\sum_{n \in F} \tilde{a}_{nk} \Big|^{p{'}}\Big)^{1 \over p{'}} \Big) \leq \|L_A \|_{\chi} \leq 4 \lim_{m \rightarrow \infty} \Big( \sup_{F \in \mathcal{F}_{m}}\Big( \sum_{k=0}^{\infty}\Big |\sum_{n \in F} \tilde{a}_{nk} \Big|^{p{'}}\Big)^{1 \over p{'}} \Big). $$
Thus
$L_A$ is compact if and only if $\displaystyle\lim_{m \rightarrow \infty} \Big( \sup_{F \in \mathcal{F}_{m}}\Big( \sum_{k=0}^{\infty}\Big |\sum_{n \in F} \tilde{a}_{nk} \Big|^{p{'}}\Big)^{1 \over p{'}} \Big) =0$.
\end{proof}

\begin{thm}
Let $ 1< p< \infty$. If $A \in (l_p(r,s,t; B), bv)$, then
$$\displaystyle\lim_{m \rightarrow \infty} \Big( \sup_{F\in \mathcal{F}_{m}}\Big( \sum_{k=0}^{\infty}\Big |\sum_{n \in F} (\tilde{a}_{nk} - \tilde{a}_{n-1,k})\Big|^{p{'}}\Big)^{1 \over p{'}} \Big) \leq \|L_A \|_{\chi} \leq 4\lim_{m \rightarrow \infty} \Big( \sup_{F \in \mathcal{F}_{m}}\Big( \sum_{k=0}^{\infty}\Big |\sum_{n \in F}( \tilde{a}_{nk} - \tilde{a}_{n-1,k})\Big|^{p{'}}\Big)^{1 \over p{'}} \Big) $$
and
\begin{center}
$L_A$ is compact if and only if $\displaystyle\lim_{m \rightarrow \infty} \Big( \sup_{F \in \mathcal{F}_{m}}\Big( \sum_{k=0}^{\infty}\Big |\sum_{n \in F} (\tilde{a}_{nk}- \tilde{a}_{n-1,k} )\Big|^{p{'}}\Big)^{1 \over p{'}} \Big) =0$.
\end{center}
\end{thm}
\begin{proof}
Using matrix domain, the sequence space $bv$ of bounded variation sequences can be written as matrix domain of triangle $\Delta$, i.e., $bv =(l_1)_{\Delta}$.
Let $A \in (l_p(r,s,t; B), bv)$. Then for each $x \in l_p(r,s,t; B)$, we get $(\Delta A)x= \Delta(Ax) \in l_1$.
So by Theorem \ref{thm12}, we have $$\displaystyle\lim_{m \rightarrow \infty} \Big( \sup_{F \in \mathcal{F}_{m}}\Big( \sum_{k=0}^{\infty}\Big |\sum_{n \in F} (\tilde{a}_{nk} - \tilde{a}_{n-1,k})\Big|^{p{'}}\Big)^{1 \over p{'}} \Big) \leq \|L_A \|_{\chi} \leq 4\lim_{m \rightarrow \infty} \Big( \sup_{F \in \mathcal{F}_{m}}\Big( \sum_{k=0}^{\infty}\Big |\sum_{n \in F}( \tilde{a}_{nk} - \tilde{a}_{n-1,k})\Big|^{p{'}}\Big)^{1 \over p{'}} \Big) $$
and
$L_A$ is compact if and only if $\displaystyle\lim_{m \rightarrow \infty} \Big( \sup_{F \in \mathcal{F}_{m}}\Big( \sum_{k=0}^{\infty}\Big |\sum_{n \in F} (\tilde{a}_{nk}- \tilde{a}_{n-1,k})\Big|^{p{'}}\Big)^{1 \over p{'}} \Big) =0$. This proves the theorem.
\end{proof}

\end{document}